\documentclass[a4paper,11pt]{amsart}
\usepackage{mathrsfs}
\usepackage[all]{xy}
\usepackage{amsmath,amssymb,amscd,bbm,amsthm,mathrsfs}
\usepackage{graphicx}
\newtheorem{thm}{Theorem}[section]
\newtheorem{lem}{Lemma}[section]
\newtheorem{cor}{Corollary}[section]
\newtheorem{prop}{Proposition}[section]
\newtheorem{rem}{Remark}[section]

\begin{document}
\numberwithin{equation}{section}

\title[   the quaternionic Monge-Amp\`{e}re operator and   plurisubharmonic functions  ]
 {\bf    the quaternionic Monge-Amp\`{e}re operator and plurisubharmonic functions  on the
Heisenberg group}
\author{  Wei Wang}
\thanks{
Supported by National Nature Science Foundation in China (No.
11571305)  }\thanks{   Department of Mathematics,
Zhejiang University, Zhejiang 310027,
 P. R. China, Email:   wwang@zju.edu.cn}

\begin{abstract} Many fundamental results of
pluripotential theory  on the quaternionic space $\mathbb{H}^n$ are extended to
  the
Heisenberg group. We introduce  notions of a plurisubharmonic function,
  the   quaternionic Monge-Amp\`{e}re operator,  differential operators $d_0$ and $d_1$  and   a closed positive current  on the
Heisenberg group. The quaternionic Monge-Amp\`{e}re operator is the coefficient  of  $ (d_0d_1u)^n$. We establish the Chern-Levine-Nirenberg type estimate, the existence of  quaternionic Monge-Amp\`{e}re
measure   for a continuous quaternionic plurisubharmonic function and the minimum principle for the   quaternionic Monge-Amp\`{e}re operator.
Unlike the tangential Cauchy-Riemann operator $ \overline{\partial}_b $   on the
Heisenberg group which behaves badly as
$
   \partial_b\overline{\partial}_b\neq -\overline{\partial}_b\partial_b
$, the  quaternionic counterpart $d_0$ and $d_1$ satisfy $
    d_0d_1=-d_1d_0
$.
   This  is the main reason  that we have a better theory for  the quaternionic Monge-Amp\`{e}re operator than   $ (\partial_b\overline{\partial}_b)^n$.
\end{abstract}
\maketitle
\section{Introduction}

  The theory of subharmonic functions (potential theory)  has already been generalized to Carnot groups in terms of SubLaplacians
 (cf. e.g. \cite{BL} \cite{DGN} and references therein), and
the generalized
horizontal Monge-Amp\`{e}re operator and $H$-convex functions on the Heisenberg group have been studied  for more than a decade  (cf. \cite{BCK} \cite{DGN} \cite{DGNT}
\cite{GT} \cite{GM}   \cite{JLMS} \cite{Ma} \cite{TZ} and references therein). For the $3$-dimensional  Heisenberg group, Guti\'errez and
Montanari \cite{GM} proved that the Monge-Amp\`{e}re measure defined by
 \begin{equation}\label{eq:MA-real}
 \int \det (Hess_X(u)) +12(T u)^2\quad {\rm for }  \quad u \in C^2(\Omega),
 \end{equation}
can be extended to $H$-convex functions, where $ Hess_X(u)$ is the symmetric $2\times2$-matrix
   \begin{equation}\label{eq:Hess-real}
      Hess_X(u):=\left(\frac {X_iX_ju+X_j X_iu}2\right)
   \end{equation}
and $X_1,X_2, T$ are standard   left invariant vector fields on the $3$-dimensional  Heisenberg group.
 $u$ is called {\it $H$-convex} on a  domain $\Omega$ if for any $\xi,\eta \in \Omega$ such that $\xi^{-1}\eta\in H_0$ and  $ \xi\delta_r(\xi^{-1}\eta)\in \Omega$ for $r\in [0,1]$,  the function of one real variable $r\rightarrow u(\xi\delta_r (\xi^{-1}\eta)) $ is convex in $  [0,1]$,  where $\delta_r$ is the dilation      and $H_0$ indicates   the subset of horizontal directions through the origin. It was generalized to  the $5$-dimensional   Heisenberg group   by Garofalo and
Tournier \cite{GT}, and to $k$-Hessian measures for $k$-convex functions on any dimensional Heisenberg groups   by Trudinger and Zhang \cite{TZ}.

In the theory of several complex variables, we have a powerful pluripotential theory about complex Monge-Amp\`{e}re operator $(\partial \overline{\partial} )^n$ and closed
positive currents, where $ \overline{\partial}  $ is the   Cauchy-Riemann operator (cf. e.g.  \cite{klimek}). It is quite interesting to develop its CR version  over the Heisenberg group.
A natural CR generalization of the  complex Monge-Amp\`{e}re operator is $(\partial_b\overline{\partial}_b)^n$, where $ \overline{\partial}_b $ is the tangential Cauchy-Riemann operator.
But unlike $\overline{\partial}\partial=-\partial\overline{\partial}$, it behaves badly as
\begin{equation}\label{eq:bad}
   \partial_b\overline{\partial}_b\neq -\overline{\partial}_b\partial_b,
\end{equation}because of the noncommutativity of horizontal vector fields (cf. Subsection 3.1). So
  it is very difficult to investigate the operator $(\partial_b\overline{\partial}_b)^n$, e.g. its regularity. On the other hand, pluripotential theory has been extended to the quaternionic  space $\mathbb{H}^n$
   (cf. \cite{alesker1}-\cite{alesker6} \cite{Bou} \cite{GMo} \cite{wanZ}-\cite{WZ}  \cite{wang-alg} and references
therein).
If we equip the $(4n+1)$-dimensional  Heisenberg group a natural
quaternionic strucuture on its horizontal subspace, we can introduce
  differential
operators $d_0$,  $d_1$ and  $\triangle u=d_0d_1u$ in terms of complex  horizontal vector fields, as the quaternionic counterpart  of $\partial_b$, $ \overline{\partial}_b $ and $\partial_b
\overline{\partial}_b $. They
  behave so well
  that we can   extend many fundamental results of   quaternionic
pluripotential theory  on $\mathbb{H}^n$ to
  the
Heisenberg group.

The $(4n+1)$-dimensional \emph{Heisenberg group}    $\mathscr{H} $ is the vector space $\mathbb{R}^{4n+1}$   with the
multiplication given by
\begin{align}\label{eq:hei}
(x,t) \cdot (y ,s )=\left( x + y,t+s +2  \langle x , y \rangle\right), \quad {\rm where}\quad
\langle x , y\rangle:=\sum_{l=1}^{2n }( x_{2l-1}y_{2l} -x_{2l} y_{2l-1})
\end{align}
for $x,y\in \mathbb{R}^{ 4n}$, ${t},{s }\in  \mathbb{R} $. Here $\langle\cdot,\cdot\rangle$ is the standard symplectic form.
We introduce a partial quaternionic  structure on the
Heisenberg group simply by identifying the underlying space of $  \mathbb{H}^n $ with $\mathbb{R}^{4n}$.
For a fixed $ q\in\mathbb{H}^n $, consider a $5$-dimensional real subspace
\begin{equation}\label{eq:Hq'}
   \mathscr{H}_q:=\{(q\lambda, t)\in \mathscr{H}; \lambda\in \mathbb{H}, t\in  \mathbb{R}\},
\end{equation} which is a subgroup.  $\mathscr{H}_q$ is nonabelian for all $ q\in\mathbb{H}^n $ except for a ${\rm codim}_{\mathbb{R}} 3$ quadratic cone $\mathfrak{D}$.
For a point $\eta\in \mathscr{H} $, the left translate of the subgroup $\mathscr{H}_q$ by $\eta$,
\begin{equation*}
\mathscr{H}_{\eta,q}:=  \eta\mathscr{H}_q,
\end{equation*} is a   $5$-dimensional real hyperplane through $\eta$,   called  a  {\it (right) quaternionic  Heisenberg   line}. A $[-\infty,\infty)$-valued  upper
semicontinuous function   on $\mathscr{H} $ is said to be \emph{plurisubharmonic}  if it is $L^1_{\rm loc} $ and is  subharmonic (in terms of SubLaplcian) on each   quaternionic  Heisenberg   line
$\mathscr{H}_{\eta,q}$ for any $\eta\in \mathscr{H} $, $ q\in\mathbb{H}^n\setminus  \mathfrak{D} $.

Let $X_1,\ldots X_{4n}$ be the standard  horizontal  left invariant vector fields    (\ref{eq:vector-Y}) on the
  Heisenberg group $\mathscr{H}  $.
Denote the {\it tangential Cauchy-Fueter operator} on $\mathscr{H}  $ by
\begin{equation*}
   \overline{Q_l}:=X_{4l+1}+\mathbf{i}X_{4l+2}+\mathbf{j}X_{4l+3 }+\mathbf{k}X_{4l+4},
\end{equation*}
and its conjugate
$
    {Q_l}=X_{4l+1}-\mathbf{i}X_{4l+2}-\mathbf{j}X_{4l+3}-\mathbf{k}X_{4l+4},
$
$  l =0,\cdots,n-1$. Compared to the Cauchy-Fueter operator on $ \mathbb{{H}}^n$, the   tangential Cauchy-Fueter operator $\overline{Q_l}$ is
much more complicated because
not  only $\mathbf{i},\mathbf{j},\mathbf{k}$ are noncommutative, but also $X_a$'s are. In particular,
\begin{equation}\label{eq:bar-QQ}
  \overline   {Q_l}{Q_l} =X_{4l+1}^2+ X_{4l+2}^2+X_{4l+3}^2+ X_{4l+4}^2 -8 \mathbf{i} \partial_{t },
\end{equation}is not real.
 But
 for a real $    C^2 $ function $u$, the $n\times n$ quaternionic matrix
   \begin{equation*}
      \left(\overline{Q_l} {Q_m} u + 8\delta_{lm} \mathbf{i} \partial_{t }u \right)
   \end{equation*}is   hyperhermitian, called  the {\it horizontal quaternionic Hessian}.
 It  is nonnegative if $u$ is plurisubharmonic.
We define the {\it quaternionic Monge-Amp\`{e}re operator} on the
Heisenberg group  as
    \begin{equation*}
   \det   \left(\overline{Q_l} {Q_m} u + 8 \delta_{lm}\mathbf{i} \partial_{t }u \right),
   \end{equation*}
   where $\det$ is the Moore determinant.

   Alesker obtained   Chern-Levine-Nirenberg estimate for the  quaternionic Monge-Amp\`{e}re operator  on $\mathbb{H}^n$ \cite{alesker2}. We extend this
   estimate to   the
Heisenberg group, and obtain the following existence theorem of the quaternionic Monge-Amp\`{e}re measure   for a continuous
plurisubharmonic function.
   \begin{thm} \label{thm:MA-measure}
   Let $\{u_j\}$ be a sequence of $C^2$ plurisubharmonic
functions converging to $u$ uniformly on compact subsets of  a
domain $\Omega$ in $\mathscr{H} $. Then $u $ be a continuous  plurisubharmonic function on $\Omega$. Moreover,  $\det\left(
{\overline{Q_l}}Q_mu_j+  8 \delta_{lm}\mathbf{i} \partial_{t }u_j\right)$ is a family of
uniformly bounded measures on each compact subset  $K$ of
$\Omega$ and weakly
converges to a non-negative measure on $\Omega$. This measure depends only on $u$ and not on
the choice of an approximating sequence  $\{u_j\}$.
\end{thm}
It is worth mentioning  that compared to  the real case (\ref{eq:Hess-real}), our quaternionic Monge-Amp\`{e}re operator need not to be
symmetrized for off-diagonal entries and the Monge-Amp\`{e}re measure does not have an extra term $(Tu)^2$ as in (\ref{eq:MA-real}).

As in \cite{shi-wang} \cite{wan-wang}  \cite{Wang}  \cite{Wang2}, motivated by the embedding of  quaternionic algebra $\mathbb{H}$ into  $\mathbb{C}^{2\times2}:$
\begin{align*}
x_{1}+x_{2}\textbf{i}_1+x_{3}\textbf{i}_2+x_{4}\textbf{i}_3\mapsto
\left(\begin{array}{rr} x_{1}+\mathbf{i}x_{2}& -x_{3}-\mathbf{i}x_{4}\\ x_{3}-\mathbf{i}x_{4}& x_{1}-\mathbf{i}x_{2}\end{array}\right),
\end{align*}
we consider complex left invariant vector fields
\begin{equation}\label{eq:nabla-jj'-new}\left(
                             \begin{array}{cc}
                               Z_{00' } & Z_{01' } \\
                                                              \vdots&\vdots\\
                               Z_{ l 0' } & Z_{ l 1' } \\
                               \vdots&\vdots\\ Z_{ n 0' } & Z_{  n 1' } \\ \vdots&\vdots\\
                               Z_{(  n+l)0' } & Z_{(  n+l)1' } \\ \vdots&\vdots\\
                                                            \end{array}
                           \right):=\left(
                                      \begin{array}{cc}
                                      X_1+\textbf{i}X_2 & -X_3-\textbf{i}X_4 \\
                                                                                \vdots&\vdots\\
                                         X_{4l+1} +\textbf{i}X_{4l+2}  & -{X_{4l+3}} -\textbf{i}{X_{4l+4}} \\ \vdots&\vdots    \\
                                         {X_{3}}-\textbf{i}{X_{4}} & {X_{1}}-\textbf{i}{X_{2}} \\   \vdots&\vdots\\
                                        {X_{4l+3}}-\textbf{i}{X_{4l+4}} & {X_{4l+1}}-\textbf{i}{X_{4l+2}} \\
                                        \vdots&\vdots
                                      \end{array}
                                    \right),
\end{equation} where $X_a$'s are the standard  horizontal  left invariant vector fields    (\ref{eq:vector-Y}) on  $\mathscr{H}  $.
Let $\wedge^{p}\mathbb{C}^{2n}$ be the complex exterior algebra generated by $\mathbb{C}^{2n}$, $  p=0,\ldots,2n$. Fix a basis
$\{\omega^0,\omega^1,\ldots$, $\omega^{2n-1}\}$ of $\mathbb{C}^{2n}$. For a domain $\Omega$   in $\mathscr H $, we
define differential operators $d_0,d_1:C_0^\infty(\Omega,\wedge^{p}\mathbb{C}^{2n})\rightarrow C_0^\infty(\Omega,\wedge^{p+1}\mathbb{C}^{2n})$ by
\begin{equation}d_0F:=\sum_{I}\sum_{A=0}^{2n-1}Z_{A0' }f_{I}~\omega^A\wedge\omega^I,\qquad d_1F:=\sum_{I}\sum_{A=0}^{2n-1}Z_{A1' }f_{I}~\omega^A\wedge\omega^I,
\end{equation}for $F=\sum_{I}f_{I}\omega^I\in C_0^\infty(\Omega,\wedge^{p}\mathbb{C}^{2n})$,  where
$\omega^I:=\omega^{i_1}\wedge\ldots\wedge\omega^{i_{p}}$ for the multi-index
$I=(i_1,\ldots,i_{p})$.  We call a form $F$
  \emph{closed} if $d_0F=d_1F=0.
$

 In contrast to the bad behaviour (\ref{eq:bad})  of
  $ \partial_b\overline{\partial}_b$, we have the following nice identities for $d_0$ and $d_1$:
  \begin{equation}\label{eq:remarkable}
    d_0d_1=-d_1d_0,
  \end{equation}
   which  is the main reason  that we could have a good theory for  the quaternionic Monge-Amp\`{e}re operator  on the
Heisenberg group.

\begin{prop}\label{prop:d2}
(1)  $d_0^2=d_1^2=0$.

(2) The identity (\ref{eq:remarkable}) holds.

(3) For $F\in C_0^\infty(\Omega,\wedge^{p}\mathbb{C}^{2n})$, $G\in C_0^\infty(\Omega,\wedge^{q}\mathbb{C}^{2n})$, we
have\begin{equation*}d_\alpha(F\wedge G)=d_\alpha F\wedge G+(-1)^{p}F\wedge d_\alpha G,\qquad \alpha=0,1.\end{equation*}
\end{prop}

We introduce a second-order differential operator $\triangle:C_0^\infty(\Omega,\wedge^{p}\mathbb{C}^{2n})\rightarrow
C_0^\infty(\Omega,\wedge^{p+2}\mathbb{C}^{2n})$ by
\begin{equation}\label{eq:triangle0}
   \triangle
F:=d_0d_1F,
\end{equation}
which behaves nicely as $ \partial\overline{\partial}$ as in the following proposition.
\begin{prop}\label{prop:d-delta}For $u_1,\ldots,
u_n\in C^2$, \begin{equation*}\begin{aligned}\triangle u_1\wedge \triangle
u_2\wedge\ldots\wedge\triangle u_n&=d_0(d_1u_1\wedge \triangle
u_2\wedge\ldots\wedge\triangle u_n)=-d_1(d_0u_1\wedge \triangle
u_2\wedge\ldots\wedge\triangle u_n)\\&=d_0d_1(u_1\triangle
u_2\wedge\ldots\wedge\triangle u_n)=\triangle (u_1
\triangle u_2\wedge\ldots\wedge\triangle u_n).
\end{aligned}\end{equation*}
\end{prop}
 The quaternionic Monge-Amp\`{e}re operator can be expressed as   the exterior product of $\triangle u$.
\begin{thm}\label{thm:delta-QMA} For a real $C^2$ function $u $  on
$\mathscr{H} $, we have
\begin{equation}\label{eq:delta-QMA}\triangle
u \wedge\ldots\wedge\triangle
u =n!~ \det\left( {\overline{Q_l}}Q_mu+ 8 \delta_{lm}\mathbf{i} \partial_{t }u\right)  \Omega_{2n},
\end{equation}
where\begin{equation}\label{eq:omega}\Omega_{2n}:=\omega^0\wedge\omega^{  n }\cdots\wedge
\omega^{ n-1} \wedge
\omega^{2n-1}\in \wedge^{2n}_{\mathbb{R}+}\mathbb{C}^{2n}.\end{equation}
\end{thm}

\begin{thm} {\rm (The minimum principle)}\label{thm:minimun}
 Let $\Omega$ be a bounded   domain with smooth boundary in $\mathscr{H} $, and let $u$ and $v $ be continuous     plurisubharmonic functions on $\Omega$. Assume that
$
  (\triangle u)^n\leq    (\triangle v)^n.
 $ Then
 \begin{equation*}
    \min_{\overline{\Omega}} \{u-v\}= \min_{ \partial {\Omega}} \{u-v\}. \end{equation*}
 \end{thm}
 An immediate corollary of this theorem is that the uniqueness of continuous solution to the Dirichlet problem for the quaternionic Monge-Amp\`{e}re equation.

Originally, we define differential operators $d_0$ and $d_1$  and the quaternionic Monge-Amp\`{e}re operator  on the right quaternionic Heisenberg group since there exist  the tangential $k$-Cauchy-Fueter complexes over this group \cite{shi-wang}. Later we find that these definitions also work on the  Heisenberg group, on which the theory is simplified because   its center is only $1$-dimensional while   the right quaternionic Heisenberg group has a $3$-dimensional center.

  This paper is arranged as follows.  In Section 2, we give preliminaries on the Heisenberg group, the group structure of right quaternionic  Heisenberg   line  $\mathscr{H}_q$, the  SubLaplacian  on  $\mathscr{H}_q$ and its fundamental solution. After   recall fundamental results on subharmonic functions on a Carnot group, we give
basic properties of plurisubharmonic functions on the
Heisenberg group. In Section 3, we   discuss operators  $d_0$,  $d_1$  and   nice behavior of brackets $      [Z_{AA '},Z_{BB '}] $, by which we can prove Proposition \ref{prop:d2}. Then we show that the horizontal  quaternionic Hessian  $(\overline{Q_l} {Q_m} u +8
\delta_{lm}\mathbf{i}\partial_t u )$ for a real $C^2 $ function $u$   is hyperhermitian, and prove the expression of the quaternionic Monge-Amp\`{e}re operator in Theorem \ref{thm:delta-QMA} by using linear algebra we developed before in \cite{wang-alg}. In Section 4, we recall definitions of  real forms and positive  forms, and show that $\triangle u$ for  a $C^2 $ plurisubharmonic function $ u$   is a closed   strongly positive
$2$-form. Then we introduce   notions  of a closed positive current  and the "integral" of a positive
$2n$-form current, and show that for any plurisubharmonic   function  $u $, $\triangle u$ is
a closed positive $2$-current. In Section 5, we give proofs of  Chern-Levine-Nirenberg estimate, the  existence of the  quaternionic Monge-Amp\`{e}re
measure   for a continuous   plurisubharmonic function  and the minimum principle.

\section{Plurisubharmonic functions over the
Heisenberg group}
\subsection{The
Heisenberg group}

  We have the following conformal transformations on $\mathscr{H} $:
(1) \emph{dilations}:
$\delta_r:
 (x,{t})\longrightarrow(r x,r^{2}{t}),$ $ r>0;
$
(2) \emph{left translations}:
$
\tau_{(y,s)}:(x,{t})\longrightarrow
(y,s)\cdot(x,{t});
$
(3) \emph{rotations}:
${R_U}:(x,{t})\longrightarrow (   Ux,t),$ for $ U\in {\rm U}(n),$
where
$
  {\rm U}(n)
$ is the unitary group;
(4) the \emph{inversion}:
$
R:(x,{t})\longrightarrow \left( \frac x{|x|^{2}+\mathbf{i}t },
\frac{ {t}}{|x|^{4}+|{t}|^{2}}\right).
$
Define
 vector fields:
\begin{equation}\label{eq:def-Y}
  X_a u( x,t):=\left.\frac d{d\varsigma} u(( x,t)(\varsigma e_a,0))\right|_{\varsigma=0},
\end{equation}on the
  Heisenberg group $\mathscr{H} $, where $e_a=(\ldots,0,1,0,\ldots)\in \mathbb{R}^{4n}$ with only $a$th entry
nonvanishing, $a=1,2,\ldots 4n$. It follows from the multiplication law (\ref{eq:hei}) that
\begin{equation} \label{eq:vector-Y}
X_{2l-1}:=\frac{\partial}{\partial x_{2l-1}}-2  x_{2l }
\frac{\partial}{\partial t },\qquad X_{2l }:=\frac{\partial}{\partial x_{2l }}+2  x_{2l-1}
\frac{\partial}{\partial t }
\end{equation}$  l= 1,\cdots, 2 n $,
  whose brackets are
\begin{equation}\label{eq:bracket-T}
[X_{2l -1},X_{2l }]=4
 \partial_{t },\quad {\rm and\hskip 2mm all\hskip 2mm  other\hskip 2mm  brackets\hskip 2mm  vanish}.
\end{equation}
 $X_a $ is {\it left invariant} in the sense that for any $(y,s)\in \mathscr{H}$,
\begin{equation}\label{eq:left-invariant}
\tau_{(y,s)*}    X_a  = X_a ,
\end{equation}by definition (\ref{eq:def-Y}), which means for fixed $(y,s)\in \mathscr{H}$,
\begin{equation}\label{eq:left-invariant2}
 \left.\left.   X_a \left(\tau_{(y,s)}^{*}f\right)\right|_{(x,t)}=(X_a  f)\right|_{( y,s)(x,t)},
\end{equation}
where the pull back function $(\tau_{(y,s)}^{*}f)(x,t) :=f((y,s)(x,t)) $. On the left hand side above, $X_a$ is the differential operator in (\ref{eq:vector-Y}) with
coefficients at point $(x,t)$, while on the right hand side, $X_a$ is the differential operator with
coefficients at point $( y,s)(x,t)$,

\subsection{Right quaternionic  Heisenberg   lines}

For  quaternionic numbers $ q,p\in \mathbb{H} $, write
\begin{equation*} q = x_1+ \mathbf{i}x_2+ \mathbf{j}x_3+ \mathbf{k}x_4 ,\qquad
  p=y_1+ \mathbf{i}y_2+ \mathbf{j}y_3+ \mathbf{k}y_4.
\end{equation*}Let $ \widehat{p}$ be the column vector in $\mathbb{R}^4$ represented by $p$, i.e.
$
     \widehat{p}:=  \left( y_1,y_{2},y_3,y_{4 }\right)^t,$
and let $q^{\mathbb{R}}$ be the $4\times4$ matrix representing the transformation of left multiplying by $q$, i.e.
\begin{equation}\label{eq:widetildeX-q}
    \widehat{qp} =  q^{\mathbb{R}}  \widehat{p} .
\end{equation}
It is direct to check (cf. \cite{wang1}) that
\begin{equation}\label{eq:X-R}
 q^{\mathbb{R}}:=     \left(
\begin{array}{rrrr }x_1&-    x_2&-  x_3&- x_4 \\    x_2&x_1&- x_4&x_3 \\  x_3&  x_4 &  x_1&-  x_2 \\x_4 &- x_3&   x_2 &x_1
\end{array}\right),
\end{equation}
 and
\begin{equation}\label{eq:iso}
    ( {q_1q_2)}^{\mathbb{R}}= q_1^{\mathbb{R}}q_2^{\mathbb{R}},\qquad ({\overline{q} )}^{\mathbb{R}} =({q  }^{\mathbb{R}})^t.
\end{equation}

The multiplication law (\ref{eq:hei}) of the   Heisenberg group can be written   as
\begin{equation}\label{eq:H-multiply} (y,{s}) \cdot
(x,t)=\left(y+ x,s  +t +2
\sum_{l=0}^{n-1}\sum_{j,k=1}^4   J_{k j}y_{4l+k}x_{4l+j} \right)
\end{equation}
with
\begin{equation}\label{eq:J-symplectic}
   J=\left(\begin{array}{rrrr}0&1&0&0\\-1&0&0&0\\0&0 &0&1\\0&0&-1&0 \end{array}\right).
\end{equation}
The multiplication of the subgroup  $ {\mathscr{H}}_{q }$  is given by
\begin{equation}\label{eq:Hq}
   (q\lambda, t)(q\lambda', t')=\left(q(\lambda+\lambda'), t+t'+2 \sum_{l=0 }^{n-1} \left( \widehat{{q_l}  {\lambda}}  \right)^t  J  \widehat{q_l     {{\lambda'}}}
    \right),
\end{equation}where
  \begin{equation}\label{eq:Bq}
  \sum_{l=0}^{n-1} \left (q_l^{\mathbb{R}}\widehat{\lambda}\right)^t   J  q_l^{\mathbb{R}} \widehat{\lambda'}=\sum_{j,k=1}^4  B^q_{kj} {\lambda}_k  {\lambda}_j',\qquad B^q:=\sum_{l=0 }^{n-1}( {q_l}^{\mathbb{R}})^t  J  q_l^{\mathbb{R}}
  \end{equation} for $\lambda=\lambda_1+\mathbf{i} \lambda_2+\mathbf{j}\lambda_3+\mathbf{k}\lambda_4,\lambda'=\lambda_1'+\mathbf{i} \lambda_2'+\mathbf{j}\lambda_3'+\mathbf{k}\lambda_4'\in \mathbb{H}$. $B^q$
  is a $4\times 4$ skew symmetric  matrix.
   So if we
consider the group $\widetilde{\mathscr{H}}_q$ as the vector space $\mathbb{R}^5$ with
the multiplication given by
\begin{equation}\label{eq:Hq0}
   (  \lambda, t)(  \lambda', t')=\left( \lambda+\lambda' , t +t' +2\sum_{ k,j=1}^4B^q_{kj}\lambda_{k }
   \lambda_{j}' \right),
\end{equation}
 we have the isomorphism of groups:
  \begin{equation}\label{eq:embedding}\begin{split}
    \iota_q: \widetilde{\mathscr{H}}_q \longrightarrow  {\mathscr{H}}_{q },\qquad
    (  \lambda, t)&\mapsto   ( q \lambda, t).
  \end{split}\end{equation} $\widetilde{\mathscr{H}}_q$  is different from the $5$-dimensional Heisenberg group in general. Note that the subgroup  $ {\mathscr{H}}_{q }$ of  $  \mathscr{H} $ is the same if $ q $ is replaced by $q q_0$ for $0\neq q_0\in\mathbb{H}$,

  Write $\textbf{i}_1:=1$, $\textbf{i}_2:=\textbf{i}$, $\textbf{i}_3:=\textbf{j}$ and $\textbf{i}_4:=\textbf{k}$.
Consider left invariant vector fields on $\widetilde{\mathscr{H}_q}$:
$
    \widetilde{X_j}u(  \lambda, t):=\left.\frac {d u}{d\varsigma}( (  \lambda, t)(\varsigma\mathbf{i}_j,0))\right|_{\varsigma =0}
$
for $(  \lambda, t)\in \widetilde{\mathscr{H}_q} $. Since
\begin{equation*}
   (  \lambda, t)(\varsigma\mathbf{i}_j,0)= \left(\cdots, \lambda_j+\varsigma,\cdots, t+2\varsigma\sum_{ k,j=1}^4B^q_{kj}\lambda_{k }\right),
\end{equation*}
we get
 \begin{equation}\label{eq:Xj}
    \widetilde{X_j}=\frac{\partial}{\partial \lambda_j}+2\sum_{ k =1}^4B^q_{kj}\lambda_{k }
\frac{\partial}{\partial t }.
  \end{equation} Define the SubLaplacian on  the   right quaternionic Heisenberg line $\widetilde{\mathscr{H}_{ q}}$ as
$
   \widetilde{\triangle_q}:=\sum_{j=1}^4\widetilde{X_j}^2.
$
Note that for $ q\in\mathbb{H}  $
$
 B^q  = \sum_{l=0 }^{n-1}  {\overline{q_l}}^{\mathbb{R}}   J  q_l^{\mathbb{R}} =- \left(\sum_{l=0 }^{n-1} \overline{q_l   } \mathbf{i}  q_l\right)^{\mathbb{R}}
$ by using  (\ref{eq:iso}) and    $\mathbf{i}^{\mathbb{R}}=-J$ by (\ref{eq:X-R}). Then
\begin{equation}\label{eq:BB}
    B^q(B^q)^t =\Lambda_q^2  I_{4\times4}, \quad {\rm where}   \quad \Lambda_q:=\left|\sum_{l=0 }^{n-1} \overline{q_l   } \mathbf{i}  q_l\right|,
\end{equation} by    (\ref{eq:iso}) again. If  write $q_l = x_{4l+1}+ \mathbf{i}x_{4l+2}+ \mathbf{j}x_{4l+3}+ \mathbf{k}x_{4l+4 }$, we have
$
   \Lambda_q^2:=  S_1^2+S_2^2+S_3^2,
 $
where
$
   S_1:=    \sum_{l=0}^{n-1} (x_{4l+1}^2+x_{4l+2}^2-x_{4l+3}^2-x_{4l+4}^2) ,$ $S_2:= 2\sum_{l=0}^{n-1} (-x_{4l+1} x_{4l+4} +  x_{4l+2} x_{4l+3}) ,$ $ S_3: =    2 \sum_{l=0}^{n-1} (x_{4l+1} x_{4l+3} +x_{4l+2} x_{4l+4}),
$
The {\it degenerate locus} $\mathfrak{D }:=\{q\in \mathbb{H}^n;\Lambda_q=0\}$   is the intersection of three quadratic hypersurfaces in $\mathbb{R}^{4n}$ given by $ S_1= S_2= S_3=0$. Thus ${\mathscr{H}}_{q }$ is abelian if and only if $q\in  \mathfrak{D }$.
\begin{prop}\label{prop:fundamental-solution} For $ q\in\mathbb{H} \setminus  \mathfrak{D }$, the fundamental solution  of $\widetilde{\triangle_q}$ on $\widetilde{\mathscr{H}}_q$ is
\begin{equation}\label{eq:fundamental-solution-q}
 \Gamma_q(\lambda,t)=-\frac {C_q}{\rho_q(\lambda,t)},  \qquad  {\rm i.e.} \quad  \widetilde{\triangle_q}\Gamma_q=\delta_0 ,
\end{equation} where
\begin{equation*}\rho_q(\lambda,t)=  \Lambda_q^2 |\lambda|^4+t^2,\qquad
  C_q^{-1}:=\int_{\widetilde{\mathscr{H}}_q} \frac { 32 \Lambda_q^2 |\lambda|^2 }{(\rho_q(\lambda,t)+1)^3}d\lambda dt.
\end{equation*}
\end{prop}
\begin{proof}
Note that for $\varepsilon>0$, we have
\begin{equation}\label{eq:derivative0}
\sum_{ j =1}^4  \widetilde{X_j}^2\frac {-1}{\rho_q+\varepsilon}=  \frac {\Sigma_{ j =1}^4 \widetilde{X_j}^2\rho_q}{(\rho_q+\varepsilon)^2} -2\frac {\Sigma_{ j =1}^4 (\widetilde{X_j}\rho_q)^2}{(\rho_q+\varepsilon)^3}.
\end{equation}
It follows from the expression (\ref{eq:Xj}) of $\widetilde{ X_j}$ that
 \begin{equation*}
  \widetilde{ X_j}\rho_q=4\Lambda_q^2 |\lambda|^2 \lambda_j +4\sum_{ k =1}^4B^q_{kj}\lambda_{k }t,
\end{equation*}and
\begin{equation}\label{eq:derivative1}\begin{split}
  \sum_{ j =1}^4  \widetilde{X_j}^2\rho_q&=4\sum_{ j =1}^4\Lambda_q^2 |\lambda|^2   +8\sum_{ j =1}^4\Lambda_q^2  \lambda_j^2 +8\sum_{ j =1}^4 \left(\sum_{ k =1}^4B^q_{kj}\lambda_{k }\right)^2\\&
 =24 \Lambda_q^2 |\lambda|^2+8\left\langle  B^q(B^q)^t\lambda, \lambda\right\rangle =32 \Lambda_q^2 |\lambda|^2 ,
  \end{split}\end{equation}
by   skew symmetry of $B^q$ and using (\ref{eq:BB}).
On the other hand, we have
\begin{equation}\label{eq:derivative2}
 \sum_{ j =1}^4 ( \widetilde{X_j}\rho_q)^2=16 \Lambda_q^4 |\lambda|^4 |\lambda|^2 +  16\Lambda_q^2   |\lambda|^2 t^2=16 \Lambda_q^2|\lambda|^2\rho_q(\lambda,t)
\end{equation}by $\sum_{ j,  k =1}^4  B^q_{k j}\lambda_k\lambda_{j }=0$.
  Substituting (\ref{eq:derivative1})-(\ref{eq:derivative2}) into (\ref{eq:derivative0}) to get
\begin{equation*}
  \sum_{ j =1}^4 \widetilde{ X_j}^2\frac {-1}{\rho_q+\varepsilon}= 32 \Lambda_q^2 \frac { |\lambda|^2\varepsilon}{(\rho_q+\varepsilon)^3}.
\end{equation*}Then $\int \varphi \widetilde{\triangle}_b (\frac {-1}{\rho_q+\varepsilon} )\rightarrow C_q^{-1}\varphi(0,0) $ for $\varphi\in C_0^\infty (\widetilde{\mathscr{H}}_q) $ by recaling and letting $\varepsilon\rightarrow0+$. We get the result.\end{proof}

\subsection{Subharmonic functions on Carnot groups} A {\it Carnot
group}  $ \mathbb{G}$  of step $r \geq 1$ is a simply connected nilpotent Lie group whose Lie
algebra $\mathfrak g$ is stratified, i.e. $\mathfrak g=\mathfrak g_1\oplus \cdots\oplus\mathfrak g_r$ and $[\mathfrak g_1, \mathfrak g_j]=\mathfrak g_{j+1}$.
Let $Y_1, \cdots ,Y_p$ are smooth left invariant vector fields on a Carnot group $ \mathbb{G}$  and
homogeneous of degree one with respect to
the dilation group of $ \mathbb{G}$, such that $\{Y_1, \cdots ,Y_p\}$ is a basis
of  $\mathfrak g_1$.
There exists a homogeneous norm $\|\cdot \|$ on a Carnot group $ \mathbb{G} $ \cite{BL} such that
\begin{equation}\label{eq:fundamental-solution}
   \Gamma(\xi,\eta) := -\frac {C_Q }{\|\xi^{-1} \eta\|^{
Q- 2 }} ,
\end{equation}
for some $Q>0$ is a fundamental solution for the {\it SubLaplacian} $\triangle_{\mathbb{G}}$  given by
$
    \triangle_{\mathbb{G}}=\sum_{j=1}^p Y_j^2,
$ (the fundamental solution used in \cite{BL} is different from the usual one (\ref{eq:fundamental-solution}) by a minus sign).
 We   denote  by $D(\xi, r)$ the  {\it ball} of center $\xi$ and radius $r$, i.e.
 \begin{equation}\label{eq:ball}
    D(\xi, r) = \{\eta \in \mathbb{G} | \|\xi^{-1}\eta\| < r \}.
 \end{equation}

Recall the
{\it representation formulae} \cite{BL} for any smooth function
$u$ on  $ \mathbb{G}$:
\begin{equation}\label{eq:representation-formulae}
   u(\xi) ={M}_r^{\mathbb{G}}(u)(\xi)-N_r(\triangle_{\mathbb{G}}u)(\xi)=\mathscr{{M}}_r^{\mathbb{G}}(u)(\xi)-\mathscr{N}_r(\triangle_{\mathbb{G}}u)(\xi),
\end{equation}for every $\xi \in\Omega$ and $r > 0 $ such that $D(\xi, r)\subset  \Omega$,
where
\begin{equation}\label{eq:meanvalue''}  \begin{split}
    {M}_r^{\mathbb{G}}(u)(\xi)&:=\frac {m_Q}{r^Q}\int_{  D(\xi, r)}K(\xi^{-1}\eta)u(\eta)dV(\eta),\\
    {N}_r^{\mathbb{G}}(u)(\xi)&:=\frac {n_Q}{r^Q}\int_0^r\rho^{Q-1}d\rho\int_{  D(\xi, \rho)}\left(\frac 1{\|\xi^{-1}\eta \|^{Q-2} }-\frac 1{\rho^{Q-2} } \right)u(\eta)dV(\eta),
 \end{split}   \end{equation}
 and
 \begin{equation}\label{eq:meanvalue'''}  \begin{split}
     \mathscr{{M}}_r^{\mathbb{G}}(u)(\xi)&:= \int_{ \partial D(\xi, r)}\mathscr{K}(\xi^{-1}\eta)u(\eta)dS(\eta), \\
    \mathscr{{N}}_r^{\mathbb{G}}(u)(\xi)&:=  {C_Q} \int_{  D(\xi, r)}\left(\frac 1{\|\xi^{-1}\eta \|^{Q-2} }-\frac 1{\rho^{Q-2} } \right)  u(\eta)dV(\eta),
 \end{split}   \end{equation}
 for   some positive constant $m_Q,n_Q$,
 and \begin{equation}\label{eq:meanvalue'} {K}=
  |\nabla_{\mathbb{G}}d|^2  ,\qquad
  \mathscr{K}=\frac
   {|\nabla_{\mathbb{G}}\Gamma|^2}{|\nabla \Gamma|}.
 \end{equation} Here $\nabla_{\mathbb{G}}$ the vector valued differential operator $(Y_1, \cdots ,Y_p)$ and $\nabla$ is the usual gradient,
  $d(\xi)=\|\xi\|$, $dV $ is the volume element and
  $dS $ is the surface measure on $\partial D(\xi, r)$. Integrals $ {{M}}_r^{\mathbb{G}}(u)$ and $\mathscr{{M}}_r^{\mathbb{G}}(u)$
  are related by the coarea formula.

A function $u$ on a domain $\Omega\subset \triangle_{\mathbb{G}}$ is called {\it harmonic} if $\triangle_{\mathbb{G}}u=0$ in the sense
of distributions.
Then a  harmonic function $u$ in an open set  $\Omega$
satisfies the {\it mean-value formula}
\begin{equation*}
   u(\xi) =\mathscr{M}_r^{\mathbb{G}}(u)(\xi)= {M}_r^{\mathbb{G}}(u)(\xi) ,
\end{equation*}by (\ref{eq:representation-formulae}).
   For an open set $ \Omega \subset \mathbb{G}$, we say that an upper semicontinuous function function $u : \Omega\rightarrow
 [-\infty,\infty)$
is {\it $\triangle_{\mathbb{G}}$-subharmonic} if for every $\xi \in \Omega$
there exists $r_\xi > 0$ such that
\begin{equation}\label{eq:submeanvalue}
   u(\xi)\leq M_r^{\mathbb{G}}(u)(\xi) \qquad {\rm for}\quad r<r_\xi .
\end{equation}

\begin{prop} {\rm (The maximum principle for the SubLaplacian \cite{BL})}
   If $\Omega\subseteq \mathbb{G}$ is a bounded open set, for every $u \in C^2(\Omega)$ satisfying $\Delta_{\mathbb{G}} u \geq 0$
in $\Omega$ and $lim sup_{\xi\rightarrow\eta} u(\xi) \leq 0$ for any $\eta\in\partial\Omega$, we have $u \leq 0$ in $\Omega$.
\end{prop}
\begin{thm}{\rm (Theorem 4.3 in \cite{BL})}\label{thm:meanvalue}
  Let $\Omega$  be an open set in $\mathbb{G}$ and $u : \Omega\rightarrow [-\infty,+\infty)$ be an
upper semicontinuous function. Then, the  following statements are equivalent:

(i) $u$ is subharmonic;

(ii) $u \in L^1_{loc}(\Omega)$, $u(\xi) = lim_{r\rightarrow 0+} M_r^{\mathbb{G}}(u)(\xi)$ for every $\xi\in\Omega$ and
$\triangle_{\mathbb{G}} u \geq 0 $ in
$\Omega$  in the   sense of distributions.
\end{thm}

 When $\mathbb{G}$ is the    Heisenberg group   $\mathscr{H} $ in (\ref{eq:hei}), the SubLaplacian is
 \begin{equation*}
    \triangle_b=\sum_{a=1}^{4n} {X}_a^2, \end{equation*}
    where $X_a$'s are given by (\ref{eq:vector-Y}). It is known that the fundamental solution of $  {\triangle}_b$ is  $-C_Q\|\cdot\|^{-Q+2}$ for
    some constant $C_Q>0$ as in Proposition \ref{prop:fundamental-solution}, with the norm   given by
  \begin{equation*}
   \|(x,t)\|:=(|x|^4+t^2)^{\frac 14}.
 \end{equation*} The invariant Haar measure  on $\mathscr{H} $ is the usual Lebesgue  measure $dxdt $ on $\mathbb{R}^{4n+3}$.
 \begin{equation}\label{eq:meanvalue-2}
 \mathscr{K}(x,t)=  \frac
   {|\nabla_{\mathbb{G}}\Gamma|^2}{|\nabla \Gamma|}(x,t)=  \frac {2C_Q(Q-2)}{\|(x,t)\|^{Q-2}}\frac {  |x|^2}{\sqrt{4|x|^6+
   t^2}},
 \end{equation}in the
 mean-value formula,
where $Q:=4n+2$, the {\it homogeneous dimension} of the $(4n+1)$-dimensional Heisenberg group $\mathscr{H} $.

When $\mathbb{G}$ is the      group   $\widetilde{\mathscr{H}}_q$ in (\ref{eq:Hq0}), the SubLaplacian is
 $  \widetilde{{\triangle}}_b$. Because of the fundamental solution of $  \widetilde{{\triangle}}_b$ given in Proposition \ref{prop:fundamental-solution}, its norm is given by
  \begin{equation*}
   \|(\lambda,t)\|_q:=(\Lambda_q^2|\lambda|^4+t^2)^{\frac 14}.
 \end{equation*} The invariant Haar measure  on $\widetilde{\mathscr{H}}_q$ is the usual Lebesgue  measure $d\lambda dt $ on $\mathbb{R}^{5}$. Its homogeneous dimension is $6$, and the
 mean-value formulae  becomes
\begin{equation}\label{eq:meanvalue-q}\begin{split}
   \mathscr{M}_r^{q }(u) (\lambda,t):&   = \int_{\partial D_q(0,r)}   \frac {|\nabla_{ q}\Gamma_{ q}|^2}{|\nabla \Gamma_{ q}|}\left((\lambda,t)^{-1}( \lambda',t')\right ) u  ( \lambda',t')  dS ( \lambda',t') ,\\
    M_r^{ q} (\eta)(\lambda,t):&= \frac {m_q}{r^6}\int_{  D_q(0,r)} K_{ q}\left((\lambda,t)^{-1}( \lambda',t')\right ) u  ( \lambda',t')  d S ( \lambda',t'),
\end{split}  \end{equation}
where $D_q(0,r)$ is the ball of radius $r$ and centered at the origin in $\widetilde{\mathscr{H}}_q$ in terms of the norm $\|\cdot\|_q$,   $m_q$ is the constant in the
representation formula (\ref{eq:meanvalue''}) for the group $\widetilde{\mathscr{H}}_q$, and by (\ref{eq:derivative2}),
\begin{equation*}
    K_{ q} (\lambda,t)=\sum_{ j =1}^4 \left( \widetilde{X_j}\rho_q^{\frac 14}\right)^2= \frac { \Lambda_q^2|\lambda|^2}{ \|(\lambda,t)\|_q^{ 2}},
\end{equation*}
which is homogeneous of degree $0$.
\subsection{Plurisubharmonic functions  on   the
Heisenberg group}

Although $\mathscr{H}_{\eta,q}$ is not a subgroup,
by the embedding
\begin{equation}\label{eq:eta}
  \iota_{\eta,q }:\widetilde{\mathscr{H}_{ q}} \rightarrow \mathscr{H}_{\eta,q},\qquad (\lambda,t)\mapsto \eta(q\lambda,t),
\end{equation}
we  say that $u$  is  {\it subharmonic function  on $\mathscr{H}_{\eta,q}$}  if $\iota_{\eta,q }^*u$ is
$\widetilde{\triangle_q}$-subharmonic   on $\widetilde{\mathscr{H}_{ q}}$. Thus, a $[-\infty,\infty)$-valued upper semicontinuous function $u$ on a domain
$\Omega\subset \mathscr{H} $ is  {\it  plurisubharmonic}  if $u$ is $L^1_{\rm loc}(\Omega)$ and $\iota_{\eta,q }^*u$ is $\widetilde{\triangle_q}$-subharmonic on
$\iota_{\eta,q }^*\Omega\cap\widetilde{\mathscr{H}_{ q}}$   for any $ q\in\mathbb{H}^n\setminus \mathfrak{D} $ and $\eta\in
\mathscr{H}^n$. Denote by $PSH(\Omega)$ the class
of all plurisubharmonic functions on $\Omega$.

Recall that the {\it convolution} of two functions $u$ and $v$ over $\mathscr{H}$ is defined as
\begin{equation*}
   u*v(x,t)=\int_{\mathscr{H} } u(y,s) v ((y,s)^{-1}(x,t))dyds.
\end{equation*}Then
\begin{equation}\label{eq:convolution-Y}
  Y( u*v)= u*Y v
\end{equation}
for any left invariant vector field $Y$ by (\ref{eq:left-invariant2}), and
\begin{equation*}
   u*v(x,t)=\int_{\mathscr{H} } u((x,t)(y,s)^{-1}) v(y,s)dyds.
\end{equation*} by taking transformation $ (y,s)^{-1}(x,t)\rightarrow (y ,s ) $ for fixed $(x,t)$, whose Jacobian   can be easily checked  to be identity.
By the non-commutativity
\begin{equation}\label{eq:non-commutativity}
   (x,t)(y,s)^{-1}\neq(y,s)^{-1}(x,t)
\end{equation}
in general, we have $u*v\neq v *u$, and
\begin{equation*}
  \partial  D(\xi, r) = \{\eta \in \mathbb{G} | \|\xi^{-1}\eta\| = r \}\neq \{\eta \in \mathbb{G} | \|\eta\xi^{-1}\| = r \}.
\end{equation*}

Consider the standard regularization given by the convolution  $\chi_\varepsilon* u$   with
\begin{equation}\label{eq:regularization}
   \chi_\varepsilon(\xi) :=\frac 1{\varepsilon^Q} \chi\left(
\delta_{\frac 1\varepsilon }( \xi)\right) ,
\end{equation}
where
$0\leq \chi\in C_0^\infty(D(0,1))$, $\int_{\mathscr{H} }\chi(\xi) dV(\xi)=1$. Then $\chi_\varepsilon* u$ subharmonic if $u$ is (cf. Proposition \ref{prop:PSH-property} (6)),
but we do not know whether  $\chi_\varepsilon* u$ is decreasing as $\varepsilon$ decreasing to $ 0$, which  we could not prove   as in the Euclidean case, because of the non-commutativity.

\begin{rem} (1) It is a consequence of Theorem \ref{thm:meanvalue} that a    function    is in $L^1_{\rm loc}(\Omega)$ if it  is $\triangle_b$-subharmonic on $ \Omega\subset \mathscr{H} $. But since $ {\mathscr{H}_{ q}}$  is different in general for different $ q\in\mathbb{H}^n\setminus \mathfrak{D} $, we do not know wether  a $PSH(\Omega)$ function   is $\triangle_b$-subharmonic on $\Omega$. So we require it as a condition in the definition.

(2) In the characterization of subharmonicity in  Theorem \ref{thm:meanvalue} there is an additional condition   $u(\xi) = lim_{r\rightarrow 0+} M_r^{\mathbb{G}}(u)(\xi)$. We know that  $  M_r^{\mathbb{G}}(u)(\xi)$ is increasing in $r$ if $\triangle_{\mathbb{G}} u \geq 0 $. By
 upper semicontinuity, this condition holds automatically if we know   $\chi_\varepsilon* u$ is decreasing as $\varepsilon$ decreasing to $ 0$.

\end{rem}

The following basic properties of PSH functions also hold on   the
Heisenberg group.

\begin{prop} \label{prop:PSH-property} Assume that $\Omega$ is a bounded domain in $\mathscr{H}$.   Then we
have that

(1) If $ u, v \in PSH(\Omega)$, then $au + bv \in  PSH(\Omega)$, for positive constants $a,b $;

(2) If $ u, v \in PSH(\Omega)$, then $\max\{ u ,v\} \in  PSH(\Omega)$;

(3) If $\{u_\alpha\}$ is a family of  locally uniformly bounded functions in $PSH(\Omega)$, then the
upper semicontinuous regularization
$(\sup_\alpha u_\alpha)^*$
is a PSH function;

(4) If $\{u_n\}$ is a   sequence  of   functions in $PSH(\Omega)$ such that $u_n $ is decreasing to
    $u\in L^1_{\rm loc}(\Omega) $, then  $ u \in PSH(\Omega)$;

(5) If $ u \in PSH(\Omega)$ and $\gamma : \mathbb{R} \rightarrow \mathbb{R }$ is   convex and nondecreasing, then
$\gamma\circ u \in PSH(\Omega)$;

(6) If $ u \in PSH(\Omega)$, then the   regularization $
\chi_\varepsilon* u(\xi)
$
is also PSH on
$\Omega'\subset \mathscr{H} $, where $\Omega'$ is subdomain such that $\Omega'D(0,\varepsilon)\subset \Omega$. Moreover, if $u$ is also
continuous, then $\chi_\varepsilon* u$ converges to $u$ uniformly on any compact subset.

(7) If $\omega\subset\subset \Omega$, $ u \in PSH(\Omega)$, $ v \in PSH(\omega)$,  and $\limsup_{\xi\rightarrow\eta} v(\xi) \leq u(\eta)$ for all $\eta\in \partial\omega$,
then the function defined by
 \begin{equation*}
   \phi=\left\{
   \begin{array}{l}
u, \qquad {\rm on}\quad \Omega \setminus\omega,\\ \max\{u, v\}, \quad {\rm on}\quad  \omega,
\end{array}
\right.
 \end{equation*}
is PSH on $\Omega$.
\end{prop}
\begin{proof} (1)-(3) follows from definition trivially.

(4) It holds since
for any fixed $ q\in\mathbb{H}^n\setminus \mathfrak{D}  $, $\eta\in \Omega$ and small $r>0$,
\begin{equation*}\begin{split}
u(\eta)&=\lim_{n\rightarrow\infty} u_n(\eta)\leq \lim_{n\rightarrow\infty}  \frac {m_q}{r^6}\int_{  D_q(0,r)} K_{ q} (\lambda,t)
    \iota_{\eta,q}^* u_n    ( \lambda ,t)   d V(\lambda,t)\\ &
  = \frac {m_q}{r^6}\int_{  D_q(0,r)} K_{ q} (\lambda,t)
    \iota_{\eta,q}^* u     ( \lambda ,t)   d V(\lambda,t)=M_r^{ q}(   u)(\eta)
\end{split}\end{equation*}
by the monotone convergence theorem.

(5) It holds since
\begin{equation*}\begin{split}
  M_r^{ q}( \gamma\circ u)(\eta)&=\frac {m_q}{r^6}\int_{  D_q(0,r)} K_{ q} (\lambda,t)
  \gamma\left( \iota_{\eta,q}^* u    ( \lambda ,t)\right)  d V(\lambda,t)\\
  &\geq \gamma\left(\frac {m_q}{r^6}\int_{  D_q(0,r)} K_{ q} (\lambda,t)
  \iota_{\eta,q}^* u    ( \lambda ,t)  d V(\lambda,t)\right)
  \geq \gamma\left(
  \iota_{\eta,q}^* u   ( 0 ,0)   \right)=( \gamma\circ u)(\eta)
\end{split}\end{equation*}
by Jensen's inequality for nondecreasing convex function $\gamma$, since  $\frac {m_q}{r^6} K_{ q} (\lambda,t)
    $ is nonnegative and its integral over $  D_q(0,r)$ is $1$. The latter fact follows from the mean value formula for the harmonic function $\equiv 1$.

(6)
For fixed $ q\in\mathbb{H}^n\setminus \mathfrak{D}  $  and $\eta\in \Omega $,
$\chi_\varepsilon *  u$ is PSH since it is smooth and
\begin{equation}\label{eq:PSH-regularization}
\begin{split}
   M_r^{ q}(\chi_\varepsilon* u)(\eta)&= \frac {m_q}{r^6}\int_{  D_q(0,r)} K_{ q} (\lambda,t) \iota_{\eta,q}^* (\chi_\varepsilon* u)(\lambda,t)  d V(\lambda,t)
   \\&= \frac {m_q}{r^6}\int_{  D_q(0,r)} K_{ q} (\lambda,t)   d V(\lambda,t)  \int_{\mathscr{H} }\chi_\varepsilon
   (y,s)  u\left((y,s)^{-1}\eta(q\lambda ,t)\right)dyds\\&
   =\int_{\mathscr{H} }\chi_\varepsilon (y,s)   M_r^{ q} \left(\iota_{(y,s)^{-1}\eta,q}^* u\right)  (0 )dyds\\&
   \geq \int_{\mathscr{H} }  \chi_\varepsilon (y,s) u\left((y,s)^{-1}\eta\right)dyds=\chi_\varepsilon* u (\eta),
\end{split}\end{equation}
by  Fubini's theorem and  subharmonicity of $u$ on the open subset  $\Omega\cap \mathscr{H}_{(y,s)^{-1}\eta,q}$. The uniform convergence is trivial.

(7) $\phi$ is obviously in $L^1_{\rm loc}(\Omega)$, and is PSH on $\mathring{\omega}$ by (2).
For   $\eta\in \partial\omega$,
\begin{equation*}
   M_r^{ q}(\phi)(\eta)\geq   M_r^{ q}(u)(\eta)\geq  u(\eta)= \phi(\eta)
\end{equation*} for small $ r>0$.
\end{proof}

\begin{rem} Our notion of plurisubharmonic functions   is different from that  introduced by Harvey and Lawson  \cite{harvey3} for   calibrated geometries, i.e.
an upper
semicontinuous function $u$ satisfies $\triangle u\geq 0$ on each calibrated submanifold in $\mathbb{R}^N$, where $\triangle$ is the Laplacian associated to the induced
Riemannian metric on  the calibrated submanifold. In our definition we require $\triangle_q u\geq 0$ for SubLaplacian $\triangle_q $ on each $5$-dimensional real hyperplane $ {\mathscr{H}_{\eta, q}}$ for $ q\in\mathbb{H}^n\setminus \mathfrak{D} $.
\end{rem}

\section{Differential operators $d_0$, $d_1$, $\triangle$  and the quaternionic Monge-Amp\`{e}re operator on the
Heisenberg group}
\subsection{Differential operators $d_0$ and  $d_1$ }
 Denote
$
    \overline{W}_j:=X_{2j-1}+\mathbf{i}X_{2j}$, $ {W}_j:=X_{2j-1}-\mathbf{i}X_{2j},
$ $ j=1,\ldots 2n.$ Then
 \begin{equation*}
    [{W}_j,\overline W_k]= 8\delta_{jk}\mathbf{i}\partial_t
 \end{equation*}
 and all other brackets vanish by (\ref{eq:vector-Y}).
  Let $\{ \ldots, \theta^{\overline{j}}, \theta^{ {j}},\ldots,\theta \}$  be the  basis dual to $\{\ldots,\overline{W}_j, {W}_j$, $\ldots, \partial_t\}$. The {\it  tangential Cauchy-Riemann operator} is defined as
$\overline{\partial}_bu=\sum \overline{W}_ju\,\theta^{\overline{j}}$. Then
\begin{equation*}\begin{split}
   \partial_b\overline{\partial}_bu&=\sum_{j,k=1}^{2n} W_k\overline{W}_ju\,\theta^k\wedge \theta^{\overline{j}},
\\
   \overline\partial_b{\partial}_bu&=\sum_{j,k=1}^{2n} \overline W_j{W}_k u\,\theta^{\overline{j}}\wedge \theta^{k}=-\sum_{j,k=1}^{2n} W_k\overline{W}_ju\theta^k\wedge\theta^{\overline{j}}+8\mathbf{i}\partial_{t }u\sum_{ k=1}^{2n} \theta^k\wedge\theta^{\overline{k}}.
\end{split}\end{equation*}
Thus $\partial_b\overline{\partial}_b\neq -\overline\partial_b{\partial}_b$.

By the definition of the operator  $\triangle$ in (\ref{eq:triangle0}), we have \begin{equation}\label{2.235}\triangle
F=\frac{1}{2}\sum_{A,B,I}(Z_{A0 '}Z_{ B 1' }-Z_{ B0' }Z_{A1' })f_{I}~\omega^A\wedge\omega^B\wedge\omega^I,
\end{equation} for $F= \sum_{ I} f_{I}~ \omega^I$.
Now for a function $u\in C^2$ we
define\begin{equation}\label{eq:triangle}\triangle_{AB}u:=\frac{1}{2}(Z_{A0' }Z_{B1' }u-Z_{B0' }Z_{A1' }u).\end{equation}
$2\triangle_{AB}$ is the determinant of $(2\times2)$-submatrix of $A$th and $ B$th rows in (\ref{eq:nabla-jj'-new}). Note that $Z_{B0'
}Z_{A1' }u$ in the above definition could not be replaced by $Z_{A1' }Z_{B0' }u$ in general  because of noncommutativity.
  Then we can write \begin{equation}\label{2.1000}\triangle u=\sum_{A,B=0}^{2n-1}\triangle_{A B}u~\omega^A\wedge
\omega^B.\end{equation}When
$u_1=\ldots=u_n=u$, $\triangle u_1\wedge\ldots\wedge\triangle u_n$
coincides with $(\triangle u)^n:=\wedge^n\triangle u$.

The following nice behavior of brackets plays a key role in the proof of properties of $d_0, d_1$.

\begin{prop}\label{prop:brackets-Z} (1) For fixed $A'=0'$ or $1'$, we have $      [Z_{AA '},Z_{BA '}]=0$ for any $A,B=0,\ldots 2n-1$, i.e. each column $\{Z_{0A '},\cdots, Z_{(2n-1)A '}\}$ in (\ref{eq:nabla-jj'-new}) spans  an abelian subalgebra.

(2) If $|A-B|\neq 0,n$, we have
\begin{equation}\label{eq:vanish}
    [Z_{A0 '},Z_{B1 '}]=0,
\end{equation}
and
   \begin{equation}\label{eq:brackets-Z101}[ Z_{l 0'}, Z_{ (n+l ) 1' } ] =
         [ Z_{(n+l )0'}, Z_{ l 1' } ] =   - 8   \mathbf{i } \partial_{t },
  \end{equation} for $l=0,,\ldots  n-1$, and
   \begin{equation}\label{eq:brackets-Z'}\begin{split} &2\triangle_{l(n+l)}=X_{4l+1}^2+ X_{4l+2}^2+X_{4l+3}^2+ X_{4l+4}^2.
   \end{split} \end{equation}
\end{prop}
\begin{proof} Noting that by (\ref{eq:nabla-jj'-new}),  $Z_{AA'}$ and $Z_{BB'}$ for $|A-B|\neq 0$ or $n$ are
linear combinations of $X_{2l+j}$'s, $j=1,2$, with different $l$, and so their bracket vanishes by (\ref{eq:bracket-T}). Thus (1) and (\ref{eq:vanish}) hold.   (\ref{eq:brackets-Z101}) follows from brackets in (\ref{eq:bracket-T}) and the expression of $Z_{AA'}$'s in (\ref{eq:nabla-jj'-new}).
(\ref{eq:brackets-Z'}) holds by
\begin{equation*}\begin{split}
   2\triangle_{l(n+l)}&=(X_{4l+1} + \mathbf{i}X_{4l+2})(X_{4l+1} -\mathbf{i}X_{4l+2})+(X_{4l+3}- \mathbf{i}X_{4l+4})( X_{4l+3} + \mathbf{i}X_{4l+4})\\& =X_{4l+1}^2+ X_{4l+2}^2+X_{4l+3}^2+ X_{4l+4}^2-\mathbf{i}[X_{4l+1} ,X_{4l+2}]+\mathbf{i}[X_{4l+3} ,X_{4l+4}]
\end{split} \end{equation*}
and using (\ref{eq:bracket-T}).
 \end{proof}

{\it Proof of Proposition \ref{prop:d2}}. (1) For any $F=\sum_If_I\omega^I$,
note that   we have $Z_{A0'} Z_{B0'}f_I=Z_{B0'}Z_{A0'}f_I $  by Proposition \ref{prop:brackets-Z} (1). So we have
\begin{equation*}d_0^2F=\sum_{ I}\sum_{A,B=0}^{2n-1}Z_{A0'}Z_{B0'}f_I~\omega^A\wedge\omega^B\wedge
\omega^I=0,\end{equation*} by  $\omega^A\wedge\omega^B =-\omega^B\wedge\omega^A.$ It is similar for $d_1^2 =0$.

(2) For any $F=\sum_If_I\omega^I$, we have
\begin{equation*}\begin{split}
d_0d_1F&=\sum_{ I}\sum_{A , B }Z_{A0' }Z_{B1' }f_{I} \omega^A\wedge\omega^B\wedge\omega^I=\quad\sum_I\sum_{|A-B|\neq 0, n}  Z_{A0'
}Z_{B1' }f_{I} \omega^A\wedge\omega^B\wedge\omega^I\\&\qquad\qquad\qquad\qquad\qquad+\sum_{ I}\sum_{l=0}^{n-1}\left(Z_{ l 0'
}Z_{(n+l )1'}-Z_{(n+l )0'} Z_{l1' }\right)  f_{I} \omega^{ l}\wedge\omega^{ n+l }\wedge\omega^I
\\&=-\sum_I\sum_{|A-B|\neq 0,n} Z_{B1' }Z_{A0' }f_{I}\omega^B\wedge \omega^A\wedge\omega^I\\&\quad-\sum_{ I}\sum_{l=0}^{n-1}(Z_{ l 1' }Z_{(n+l )0'}-Z_{(n+l )1'} Z_{l0' })  f_{I} \omega^{ l}\wedge\omega^{ n+l
}\wedge\omega^I
\\& =-\sum_{A,B,I}Z_{A1'}Z_{B 0' }f_{I} \omega^A\wedge\omega^B\wedge\omega^I =-d_1d_0F,
 \end{split} \end{equation*}
 by using commutators    (\ref{eq:vanish})-(\ref{eq:brackets-Z101})  in Proposition \ref{prop:brackets-Z}   in the third identity.

(3) Write $G=\sum_Jg_J\omega^J$. We have
\begin{equation*}\begin{aligned}d_\alpha(F\wedge G)
=&\sum_{A,I,J}[Z_{A\alpha'}(f_I)g_J+f_IZ_{A\alpha'}(g_J)]~\omega^A\wedge\omega^I\wedge\omega^J\\
=&\sum_{A,I }Z_{A\alpha'}(f_I)~\omega^A\wedge\omega^I\wedge \sum_{ J} g_J\omega^J+(-1)^{p}\sum_{A,I }f_I\omega^I\wedge\sum_{ J}Z_{A\alpha'
}(g_J)\omega^A\wedge\omega^J\\
=&d_\alpha F\wedge G+(-1)^{p}F\wedge d_\alpha G.
\end{aligned}\end{equation*} by $\omega^A\wedge\omega^I =(-1)^{p}\omega^I\wedge\omega^A $.\hskip 113mm $\Box$

\begin{cor}\label{p2.33}For $u_1,\ldots,
u_n\in C^2$, $\triangle
u_1\wedge\ldots\wedge\triangle u_k$ is closed, $k=1,\ldots,n,$.\end{cor}
\begin{proof}By Proposition \ref{prop:d2} (3), we have \begin{equation*}d_\alpha(\triangle u_1\wedge\ldots\wedge\triangle
u_k)=\sum_{j=1}^k\triangle u_1\wedge\ldots\wedge d_\alpha(\triangle
u_j)\wedge\ldots\wedge\triangle u_k,
\end{equation*}for $\alpha=0,1$. Note that $d_0\triangle=d_0^2d_1=0$ and $
d_1\triangle=-d_1^2d_0=0 $ by using Proposition \ref{prop:d2} (1)-(2). It follows that $d_\alpha(\triangle
u_1\wedge\ldots\wedge\triangle
u_k)=0$.
\end{proof}

{\it Proof of Proposition \ref{prop:d-delta}}.
 It follows from Corollary \ref{p2.33} that\begin{equation*}d_0(\triangle
u_2\wedge\ldots\wedge\triangle u_n)=d_1(\triangle
u_2\wedge\ldots\wedge\triangle u_n)=0.
\end{equation*}By Proposition \ref{prop:d2} (3), \begin{equation*}\begin{aligned}d_\alpha(u_1\triangle
u_2\wedge\ldots\wedge\triangle u_n)&=d_\alpha u_1\wedge\triangle
u_2\wedge\ldots\wedge\triangle u_n+u_1d_\alpha(\triangle
u_2\wedge\ldots\wedge\triangle u_n)\\&=d_\alpha u_1\wedge\triangle
u_2\wedge\ldots\wedge\triangle u_n.
\end{aligned}\end{equation*}So we have
\begin{equation*}\begin{aligned}\triangle (u_1
\triangle u_2\wedge\ldots\wedge\triangle u_n)=&d_0d_1(u_1\triangle
u_2\wedge\ldots\wedge\triangle u_n)=d_0(d_1u_1\wedge\triangle
u_2\wedge\ldots\wedge\triangle u_n)\\=&d_0d_1u_1\wedge\triangle
u_2\wedge\ldots\wedge\triangle u_n-d_1u_1\wedge d_0(\triangle
u_2\wedge\ldots\wedge\triangle u_n)\\=&\triangle u_1\wedge \triangle
u_2\wedge\ldots\wedge\triangle u_n.
\end{aligned}\end{equation*}

\subsection{The quaternionic Monge-Amp\`{e}re operator   on the   Heisenberg groups}A quaternionic $(n\times n)$-matrix
$(\mathcal{M}_{jk})$ is called {\it hyperhermitian} if  ${\mathcal{M}}_{jk}=\overline{\mathcal{M}_{kj}}$.
\begin{prop} \label{prop:diagonal} {\rm (Claim 1.1.4, 1.1.7  in \cite{alesker1})} For a hyperhermitian $(n\times n)$-matrix
$\mathcal{M}$, there exists a unitary matrix $\mathcal{U}$ such that $\mathcal{U}^*\mathcal{M}\mathcal{U}$ is diagonal and real.
\end{prop}
\begin{prop} \label{prop:Moore} {\rm (Theorem 1.1.9 in \cite{alesker1})} (1) The Moore determinant of  any complex hermitian  matrix
considered as a quaternionic hyperhermitian matrix  is equal to its usual determinant.

(2) For any quaternionic hyperhermitian $(n\times n)$-matrix $\mathcal{M}$ and any quaternionic $(n\times n)$-matrix $\mathcal{C}$
\begin{equation*}
     \det(\mathcal{C}^*\mathcal{M}\mathcal{C})=\det(\mathcal{ A} )\det(\mathcal{C}^* \mathcal{C}).
\end{equation*}
\end{prop}

\begin{prop}\label{prop:hyperhermitian} For a real $C^2 $ function $u$,  the horizontal  quaternionic Hessian  $(\overline{Q_l} {Q_m} u +8
\delta_{lm}\mathbf{i}\partial_t u )$ is hyperhermitian.
\end{prop}
\begin{proof} It follows from definition (\ref{eq:nabla-jj'-new}) of $Z_{AA'}$'s that
\begin{equation}\label{eq:j-Z}
   \textbf{j}Z_{(n+m )0'}=-Z_{ m 1'}\textbf{j},\qquad \textbf{j}Z_{(n+m )1'}=Z_{ m 0'}\textbf{j}
\end{equation}and so
\begin{equation}\label{eq:QlQm0}\begin{aligned}\overline{Q_l} {Q_m}
&=\left(X_{4l+1}+\mathbf{i}X_{4l+2}+\mathbf{j}X_{4l+3}+\mathbf{k}X_{4l+4}\right)
\left(X_{4m+1}-\mathbf{i}X_{4m+2}-\mathbf{j}X_{4m+3}-\mathbf{k}X_{4m+4}\right)
\\&=\left(Z_{ l 0'}-Z_{l 1'}\textbf{j}\right)\left(Z_{(n+m )1'}-\textbf{j}Z_{(n+m )0'}\right)  \\&=\left(Z_{ l 0'}Z_{(n+m )1'}-Z_{ l
1'}Z_{(n+m )0'}\right) +\left(Z_{ l  0'}Z_{ m  1'}-Z_{ l 1'}Z_{ m 0'}\right) \textbf{j}.  \end{aligned}\end{equation}

When $l=m$, it follows    that
\begin{equation}\begin{aligned}\overline{Q_l} {Q_l} u
&
 =2 \triangle_{ l(n+l )}u -[Z_{ l 1'},Z_{(n+l )0'}]u+[Z_{ l  0'},Z_{ l  1'}] u\textbf{j}
=2 \triangle_{ l(n+l )}u - 8 \mathbf{i} \partial_{t }  u
\end{aligned}\end{equation}
by using (\ref{eq:vanish})-(\ref{eq:brackets-Z101}). Thus $\overline{Q_l} {Q_l} u +  8 \mathbf{i} \partial_{t }u$ is
real by (\ref{eq:brackets-Z'}).
If $l\neq m$,
we have
\begin{equation}\label{eq:QlQm}\begin{aligned}\overline{Q_l} {Q_m} u &=\left(Z_{ l 0'}Z_{(n+m )1'}-Z_{(n+m )0'}Z_{ l
1'}\right)u+\left(Z_{ l  0'}Z_{ m  1'}-Z_{ m 0}Z_{ l 1'}\right)u\textbf{j}\\
&=2\left(\triangle_{ l(n+m )}u+\triangle_{ l  m }u\textbf{j}\right), \end{aligned}\end{equation}
by using commutators
\begin{equation}\label{eq:commutators-1}
   [Z_{ l 1'},Z_{(n+m )0'}]=0 \qquad  {\rm and} \qquad  [Z_{ m 0'}, Z_{ l 1'}]=0,\qquad   {\rm for }\qquad  l\neq m,
\end{equation}
 by Proposition \ref{prop:brackets-Z}.

 To see the   horizontal  quaternionic Hessian  to be hyperhermitian, note that for $l\neq m$
  \begin{equation}\label{eq:Q-hyper}\begin{aligned}\overline{\overline{Q_l} {Q_m}  u } =2\left(\overline{\triangle_{ l(n+m
 )}u}-\textbf{j}\overline{ {\triangle}_{lm} u }\right),
 \end{aligned}\end{equation}
and
 \begin{equation}\label{eq:delta-anti}\begin{split}
    \overline{ {\triangle}_{l(n+m)} {u}}&=\overline{Z_{ l 0'}}\overline{Z_{(n+m )1'}}u-\overline{Z_{(n+m)0'}}\overline{Z_{ l 1'}}u
   = {Z_{(n+l )1'}}{Z_{ m 0'}}u -{Z_{ m 1'}} {Z_{(n+l )0'}}u \\&
   ={Z_{ m 0'}} {Z_{(n+l )1'}}u -{Z_{(n+l )0'}}{Z_{ m 1'}}u=  {\triangle}_{m(n+l )}  {u}
\end{split}\end{equation} by the conjugate of $Z_{AA'}$'s in (\ref{eq:nabla-jj'-new}) and (\ref{eq:commutators-1}).
 Similarly, for any $l,m$, we have
 \begin{equation}\label{eq:delta-anti2}\begin{split}
   \overline{ {\triangle}_{lm} u }&= (\overline{Z_{ l 0'}}\overline{Z_{ m 1'}}-\overline{Z_{ m 0'}}\overline{Z_{ l
  1'}})u=  ( -{Z_{ (n+l )  1'}}{Z_{ (n+m ) 0'}}+{Z_{ (n+m ) 1'}} {Z_{(n+l )  0'}})u,
\end{split}\end{equation}
 and so
\begin{equation}\label{eq:delta-anti2'}\begin{split}
  \textbf{j}\overline{ {\triangle}_{lm} u }
  &=(  {Z_{  l 0'}}{Z_{ m 1'}}-{Z_{  m 0'}} {Z_{ l 1'}})u\,\textbf{j} = - {\triangle}_{ m  l }u\, \textbf{j}
\end{split}\end{equation}by using (\ref{eq:j-Z}) and (\ref{eq:commutators-1}). Now
substitute (\ref{eq:delta-anti}) and (\ref{eq:delta-anti2'}) into  (\ref{eq:Q-hyper}) to get
 \begin{equation*} \begin{aligned}\overline{\overline{Q_l} {Q_m}  u } =2\left( \triangle_{m(n+l
 )}u + {\triangle}_{m l}u \textbf{j} \right)= {\overline {Q_m}{Q_l} u }
 \end{aligned}\end{equation*} for $l\neq m$.
This together with the reality of $\overline{Q_l} {Q_l} u +  8 \mathbf{i} \partial_{ t }u$    implies that the quaternionic Hessian
$(\overline{Q_l} {Q_m} u + 8\delta_{lm}  \mathbf{i} \partial_{ t }u )$ is hyperhermitian.
\end{proof}

As in \cite{wang-alg},
denote by $M_{\mathbb{F}}(p,m)$ the space of $\mathbb{F}$-valued $(p\times m)$-matrices, where
$\mathbb{F}=\mathbb{R},\mathbb{C},\mathbb{H}$.
For a quaternionic $p\times m$-matrix $\mathcal{M}$, write $\mathcal{M}=a+b\mathbf{j}$ for some    complex matrices $a, b\in
M_{\mathbb{C}}( p, m)$.
Then we define the  $ {\tau}(\mathcal{M})$ as the complex $(2p\times 2m)$-matrix
\begin{equation}\label{eq:tau-A}
    {\tau}(\mathcal{M}):= \left(
                                       \begin{array}{rr}
                                          {a} &-{b} \\
                                         \overline {b} & \overline{{a}} \\
                                       \end{array} \right),
\end{equation}
 Recall that for  skew symmetric matrices $M_\alpha=(M_{\alpha;ij})\in M_{\mathbb{C}}(2n, 2n )$, $\alpha=1,\ldots, n$, such that  $2$-forms
 $\omega_\alpha=\sum_{i,j} M_{\alpha;ij}\omega^i\wedge\omega^j$  are real, define
\begin{equation}\label{eq:triangleup-n}
   \omega_1\wedge\cdots\wedge\omega_n=\bigtriangleup_n(M_1,\ldots,M_n)\Omega_{2n},
\end{equation}

Consider the homogeneous polynomial $\text{det}(\lambda_1\mathcal{M}_1+\ldots+\lambda_n\mathcal{M}_n)$ in real variables
$\lambda_1,\ldots,\lambda_n$ of degree $n$. The coefficient of the monomial $\lambda_1\ldots\lambda_n$ divided by $n!$ is called the
{\it mixed discriminant} of the matrices $\mathcal{M}_1,\ldots,\mathcal{M}_n$, and it is denoted by
$\text{det}(\mathcal{M}_1,\ldots,\mathcal{M}_n)$. In particular, when $\mathcal{M}_1=\ldots=\mathcal{M}_n=\mathcal{M }$,
$\text{det}(\mathcal{M}_1,\ldots,\mathcal{M}_n)=\text{det}(\mathcal{M })$.

 \begin{thm} \label{thm:det} {\rm (Theorem 1.2 in  \cite{wang-alg})} For    hyperhermitian matrices $\mathcal{M}_1,\ldots,
 \mathcal{M}_n\in M_{\mathbb{ H}}( n )$,   we have
    \begin{equation}\label{eq:det}2^n n!~\text{det}~( \mathcal{M}_1,\ldots,\mathcal{M}_n)=
       \triangle_n\left( \tau(\mathcal{M}_1)\mathbb{J},\ldots, \tau(\mathcal{M}_n)\mathbb{J } \right),
    \end{equation}
    where\begin{equation}\label{eq:j-symplectic}
   \mathbb{J}=\left(
                                       \begin{array}{cc}
                                           0& I_n \\
                                      -I_n& 0\\
                                       \end{array} \right).
\end{equation}
 \end{thm}

{\it Proof of Theorem \ref{thm:delta-QMA}}.  The proof is similar to that of Theorem 1.3 in  \cite{wang-alg} except that $Z_{AA'}$'s
are noncommutative.
   (\ref{eq:QlQm}) implies that the   quaternionic Hessian  can be written as
\begin{equation*}
\left (\overline{Q_l} {Q_m} u +8\delta_{lm}\mathbf{i}\partial_t u\right ) =a+b\mathbf{j},
\end{equation*}
with $n\times n$ complex matrices
\begin{equation*}
   a=2(\triangle_{ l(n+m )}u),\qquad b=2(\triangle_{ l   m  }u).
\end{equation*}
Thus
\begin{equation}\label{eq:Q-delta}\begin{split}
   \tau\left (\overline{Q_l} {Q_m} u+ 8\delta_{lm}\mathbf{i}\partial_t u \right)\mathbb{J}&=\left(
                                       \begin{array}{rr}
                                         a &- b \\
\overline{b}&\overline{a}\\
                                       \end{array} \right)\mathbb{J}=\left(
                                       \begin{array}{rr}
                                         b &  a \\
-\overline{a}&\overline{b}\\
                                       \end{array} \right)\\&=2\left(
                                       \begin{array}{rr}
                                        \triangle_{ l   m  }u &\triangle_{ l(n+m )}u \\
- \overline{\triangle_{ l(n+m )}u}&\overline{\triangle_{ l m }u}\\
                                       \end{array} \right) ,
\end{split}\end{equation} Note that $\triangle_{ l   l }u= \triangle_{ (n+l )  (n+l)   }{u}=0$ by definition. For $l\neq m$,
\begin{equation*}
    \overline{ {\triangle}_{l(n+m)}  }
   ={Z_{ m 0'}} {Z_{(n+l )1'}} -{Z_{(n+l )0'}}{Z_{ m 1'}}= - {\triangle}_{(n+l )m}
\end{equation*}
by (\ref{eq:delta-anti}),  while for $l= m$  we also have
\begin{equation}\label{eq:delta-anti'}\begin{split}
    \overline{ {\triangle}_{l(n+l)} {u}}&={Z_{(n+l )1'}}{Z_{ l 0'}} u-{Z_{ l 1'}} {Z_{(n+l )0'}}u
      ={Z_{ l 0'}}{Z_{(n+l )1'}} u-{Z_{(n+l )0'}}{Z_{ l 1'}}u = - {\triangle}_{(n+l )l}  {u} ,
\end{split}\end{equation} by using Proposition \ref{prop:brackets-Z} (2).
Moreover,
\begin{equation*}
   \overline{ {\triangle}_{lm} {u}} =\triangle_{ (n+l )  (n+m )   }{u},
\end{equation*}
which follows from
(\ref{eq:delta-anti2}).
 Therefore we have
 \begin{equation}\label{eq:Q-delta'}\begin{split}
   \tau\left (\overline{Q_l} {Q_m} u+ 8\delta_{lm}\mathbf{i}\partial_tu \right)\mathbb{J}&=2 (\triangle_{ AB  }u).
\end{split}\end{equation}
 Then the result follows from applying Theorem \ref{thm:det} to matrices $\mathcal{M}_j=2\left(\overline{Q_l} {Q_m} u_j +
 8\delta_{lm}\mathbf{i}\partial_tu_j\right)$.

\section{   closed positive currents  on the quaternionic
Heisenberg group}

\subsection{Positive
$2k$-forms} Now let us recall definitions of real forms and positive $2k$-forms  (cf. \cite{alesker2} \cite{wan-wang} \cite{wang-alg}  and references therein).
Fix a basis
$\{\omega^0,\omega^1,\ldots,\omega^{2n-1}\}$ of $\mathbb{C}^{2n}$. Let
 \begin{equation}\label{eq:bn}
 \beta_n:=\sum_{l=0}^{n-1} \omega^l\wedge\omega^{  n+l } .
\end{equation}
Then
$ \beta_n^n=\wedge^n\beta_n=n!~\Omega_{2n},$ where $\Omega_{2n}$ is given by (\ref{eq:omega}). For
$\mathcal{A}\in\text{GL}_{\mathbb{H}}(n)$, define the \emph{induced $\mathbb{C}$-linear transformation} of $\mathcal{A}$ on $\mathbb{C}^{2n}$ as $\mathcal{A}.\omega^p=\tau(\mathcal{A}). \omega^p$ with
\begin{equation}\label{A.}M.\omega^p=\sum_{j=0}^{2n-1}M_{pj} \omega^j,\end{equation}
for $M\in M_{\mathbb{C}}(2n,2n)$, and define the \emph{induced
$\mathbb{C}$-linear transformation} of $\mathcal{A}$ on $\wedge^{2k}\mathbb{C}^{2n}$ as
 \begin{equation*}
   \mathcal{ A}.(\omega^0\wedge\omega^1\wedge\ldots\wedge\omega^{2k-1})=\mathcal{A}.\omega^0\wedge \mathcal{A}.\omega^1\wedge\ldots \wedge \mathcal{A}.\omega^{2k-1}.
 \end{equation*}
Therefore for $A\in\text{U}_{\mathbb{H}}(n)$,
$\mathcal{A}.\beta_n  =\beta_n,
$. Consequently $\mathcal{A}.(\wedge^n
\beta_n)=\wedge^n \beta_n$, i.e.,
$\mathcal{A}.\Omega_{2n}=\Omega_{2n}$.

$\textbf{j}$ defines a real linear map
\begin{equation}\label{eq:rouj}\rho(\textbf{j}):\mathbb{C}^{2n}\rightarrow\mathbb{C}^{2n},
\qquad\rho(\textbf{j})(z\omega^k)=\overline{z}\mathbb{\mathbb{J}}.\omega^k,
\end{equation}which is not $\mathbb{C}$-linear, where $\mathbb{J}$ is given by (\ref{eq:j-symplectic}).   Also the right multiplying of $\textbf{i}$:
$(q_1,\ldots,q_n)\mapsto(q_1\textbf{i},\ldots,q_n\textbf{i})$ induces
\begin{equation*}\label{roui}\rho(\textbf{i}):\mathbb{C}^{2n}\rightarrow\mathbb{C}^{2n},
\qquad\rho(\textbf{i})(z\omega^k)=z\textbf{i}\omega^k.
\end{equation*}Thus $\rho$ defines $GL_{\mathbb{H}}(1)$-action on $\mathbb{C}^{2n}$. The actions of $GL_{\mathbb{H}}(1)$ and
$GL_{\mathbb{H}}(n)$ on $\mathbb{C}^{2n}$ are commutative, and equip  $\mathbb{C}^{2n}$ a structure of
$GL_{\mathbb{H}}(n)GL_{\mathbb{H}}(1)$-module. This action extends to $\wedge^{2k}\mathbb{C}^{2n}$ naturally.

The real action (\ref{eq:rouj}) of $ \rho(\textbf{j})$ on $\mathbb{C}^{2n}$ naturally induces an action on $\wedge^{2k}\mathbb{C}^{2n}$.
An element $\varphi$ of $\wedge^{2k}\mathbb{C}^{2n}$ is called  {\it real} if $\rho(\textbf{j})\varphi=\varphi$. Denote by
$\wedge^{2k}_{\mathbb{R}}\mathbb{C}^{2n}$ the subspace of all real elements in $\wedge^{2k}\mathbb{C}^{2n}$. These forms are
counterparts of $(k,k)-$forms in complex analysis.

A right $\mathbb{H}$-linear
map $g:\mathbb{H}^{k}\rightarrow \mathbb{H}^{m}$ induces a $\mathbb{C}$-linear map $\tau(g):\mathbb{C}^{2k}\rightarrow\mathbb{C}^{2m}$.
If we write $g=(g_{jl})_{m\times k}$ with $g_{jl}\in \mathbb{H}$, then $\tau(g)$ is the complex $(2m\times2k)$-matrix  given by
(\ref{eq:tau-A}). The induced $\mathbb{C}$-linear pulling back transformation of
$g^*:\mathbb{C}^{2m}\rightarrow\mathbb{C}^{2k}$ is defined
as:\begin{equation}\label{g^*}g^*\widetilde{\omega}^p=\tau(g).\omega^p=\sum_{j=0}^{2k-1}\tau(g)_{pj}\omega^j,\qquad p=0,\ldots,2m-1,\end{equation}where
$\{\widetilde{\omega}^0,\ldots,\widetilde{\omega}^{2m-1}\}$ is a basis of $\mathbb{C}^{2m}$ and $\{\omega^0,\ldots,\omega^{2k-1}\}$ is a
basis of $\mathbb{C}^{2k}$. It  induces a $\mathbb{C}$-linear pulling back transformation  on $\wedge^{2k}\mathbb{C}^{2m}$   given by
$
  g^*(\alpha\wedge\beta)=g^*\alpha\wedge g^*\beta
$ inductively.

An element $\omega\in\wedge_{\mathbb{R}}^{2k}\mathbb{C}^{2n}$ is said to be \emph{elementary strongly positive} if there exist linearly
independent right $\mathbb{H}$-linear mappings $\eta_j:\mathbb{H}^n\rightarrow \mathbb{H}$ , $j=1,\ldots,k$, such that
\begin{equation*}\omega=\eta_1^*\widetilde{\omega}^0\wedge
\eta_1^*\widetilde{\omega}^1\wedge\ldots\wedge\eta_k^*\widetilde{\omega}^0\wedge \eta_k^*\widetilde{\omega}^1,\end{equation*}where
$\{\widetilde{\omega}^0,\widetilde{\omega}^1\}$ is a basis of $\mathbb{C}^{2}$ and $\eta_j^*:~\mathbb{C}^{2}\rightarrow\mathbb{C}^{2n}$
is the induced $\mathbb{C}$-linear pulling back transformation of $\eta_j$.  The definition in
the case $k=0$ is obvious:
$\wedge^{0}_{\mathbb{R}}\mathbb{C}^{2n}=\mathbb{R}$ and the positive elements
are the usual ones. For $k=n$, dim$
_{\mathbb{C}}\wedge^{2n}\mathbb{C}^{2n}=1$,  $\Omega_{2n}$ defined by (\ref{eq:omega}) is an element
of $\wedge_{\mathbb{R}}^{2n}\mathbb{C}^{2n}$ ($\rho(\textbf{j})\beta_n=\beta_n$) and spans it. An
element $\eta\in\wedge_{\mathbb{R}}^{2n}\mathbb{C}^{2n}$ is called\emph{
positive} if $\eta=\kappa~\Omega_{2n}$ for some non-negative number $\kappa$. By definition,
   $\omega\in\wedge_{\mathbb{R}}^{2k}\mathbb{C}^{2n}$ is elementary   strongly positive  if and only if
\begin{equation}\label{eq:elementary-strongly-positive}
  \omega=\mathcal{M}.(\omega^0\wedge\omega^n\wedge\ldots\wedge\omega^{ k-1}\wedge\omega^{ n+k-1})
\end{equation}for some quaternionic matrix $\mathcal{M}\in M_{\mathbb{H}}(k,n)$ of rank $k$.

An element $\omega\in\wedge_{\mathbb{R}}^{2k}\mathbb{C}^{2n}$ is called \emph{strongly positive} if it belongs to the convex cone
$\text{SP}^{2k}\mathbb{C}^{2n}$ in $\wedge_{\mathbb{R}}^{2k}\mathbb{C}^{2n}$ generated by elementary strongly positive $(2k)$-elements; that
is, $\omega=\sum_{l=1}^m\lambda_l\xi_l$ for some non-negative numbers $\lambda_1,\ldots,\lambda_m$ and some elementary strongly positive
elements $\xi_1,\ldots,\xi_m$. An $2k$-element $\omega$ is said to be \emph{positive} if for any   strongly positive element
$\eta\in \text{SP}^{2n-2k}\mathbb{C}^{2n}$, $\omega\wedge\eta$ is positive. We will denote the set of all positive $2k$-elements by
$\wedge^{2k}_{\mathbb{R}+}\mathbb{C}^{2n}$. Any $2k$ element is a $\mathbb{C}$-linear
combination of strongly positive $2k$ elements by Proposition 5.2 in \cite{alesker2}, i.e.  ${\rm span}_\mathbb{C}\{\varphi;~\varphi\in
\wedge^{2k}_{\mathbb{R}+}\mathbb{C}^{2n}\}=span_\mathbb{C}\{\varphi;~\varphi\in
\text{SP}^{2k}\mathbb{C}^{2n}\}=\wedge^{2k}\mathbb{C}^{2n}$.
By definition, $ \beta_n$ is a strongly  positive
$2$-form, and
$ \beta_n^n=\wedge^n\beta_n=n!~\Omega_{2n}$ is a positive
$2n$-form.

For a domain $\Omega$   in $\mathscr{H} $, let
$\mathcal{D}_0^{p}(\Omega)=C_0(\Omega,\wedge^{p}\mathbb{C}^{2n})$ and $
\mathcal{D}^{p}(\Omega)=C_0^\infty(\Omega,\wedge^{p}\mathbb{C}^{2n}).
$ An element of the latter one is called a \emph{test
$p$-form}. An element $\eta\in\mathcal{D}_0^{2k}(\Omega)$ is called a \emph{positive $2k$-form} (respectively, \emph{strongly positive
$2k$-form}) if for any $q\in\Omega$, $\eta(q)$ is a positive (respectively, strongly positive) element.

Theorem 1.1 in \cite{wang-alg} and its proof implies the following result.
\begin{prop} \label{prop:normal-form} For a hyperhermitian $n\times n$-matrix $\mathcal{M}=({\mathcal{M}}_{jk})$,
there exists a quaternionic unitary matrix $\mathcal{E}\in \text{U}_{\mathbb{H}}(n)$ such that
$\mathcal{E}^*\mathcal{M}\mathcal{E}={\rm diag} (\nu_0,\ldots,\nu_{n-1})$. Then
the $2$-form
\begin{equation}\label{eq:real-2form}
   \omega =\sum_{A,B=0}^{2n-1} M_{ AB}\,\omega^A\wedge\omega^B,
\end{equation} with
$
   M=\tau(\mathcal{M})\mathbb{J},
$ can normalize $\omega$ as
\begin{equation}\label{eq:normal-form}
                                         \omega =2\sum_{ l=0}^{n-1} \nu_{ l}\widetilde{\omega}^l\wedge\widetilde{\omega}^{l+n}
                                       \end{equation}
    with $\widetilde {\omega}^A= \mathcal{E^*}.{\omega}^A$. In particular, $\omega$ is strongly positive if and only if $\mathcal{M} $ is nonnegative.
\end{prop}

\begin{prop} \label{prop:wedge-positive}   For any $C^1$ real function $u$,  $d_0u\wedge d_1u$ is elementary strongly positive if grad $u\neq0$.
 \end{prop}
 \begin{proof}
Let  $p:=(p_1,\ldots,p_n)\in\mathbb{ H}^n$ with
 $
    p_l= X_{4l+1}u +\mathbf{i}X_{4l+2}u +\mathbf{j}X_{4l+3}u +\mathbf{k}X_{4l+4}u .$
 Then as \eqref{eq:QlQm0}, we have
 \begin{equation}
 \overline{p_l} {p_m} =\widetilde{ \triangle}_{l(n+m)} +\widetilde{ \triangle}_{l m } \textbf{j},
 \end{equation}
where
\begin{equation*}
  \widetilde{ \triangle}_{AB}:=Z_{ A 0'}uZ_{B1'}u-Z_{ B
1'}uZ_{A0'}u.
\end{equation*}
 Denote $n\times n$ quaternionic matrix $\widetilde{\mathcal{M}}:=(\overline{p_l} {p_m})$. Then  $\widetilde{\mathcal{M}} =a+b\mathbf{j} $
 with $n\times n$ complex matrices
$
   a=(\widetilde{\triangle}_{ l(n+m )}u),$ $ b=(\widetilde{\triangle}_{ l   m  }u).
$
Thus
\begin{equation}\label{eq:p-delta}\begin{split}
   \tau\left ( \widetilde{\mathcal{M}}\right)\mathbb{J}&=\left(
                                       \begin{array}{rr}
                                         a &- b \\
\overline{b}&\overline{a}\\
                                       \end{array} \right)\mathbb{J}=\left(
                                       \begin{array}{rr}
                                         b &  a \\
-\overline{a}&\overline{b}\\
                                       \end{array} \right) =\left(
                                       \begin{array}{rr}
                                        \widetilde{\triangle}_{ l   m  }  &\widetilde{\triangle}_{ l(n+m )}  \\
- \overline{\widetilde{\triangle}_{ l(n+m )} }&\overline{\widetilde{\triangle}_{ l m } }\\
                                       \end{array} \right)= (\widetilde{\triangle}_{ AB  } ) ,
\end{split}\end{equation}
 since we can easily check
 \begin{equation*} \begin{split}
    \overline{ \widetilde{ \triangle}_{l(n+m)}   }  = - \widetilde{{\triangle}}_{(n+ l )m},\qquad    \overline{ \widetilde{ \triangle}_{l m }   }  = \widetilde{{\triangle}}_{(n+ l )(n+m ) }.
\end{split}\end{equation*}  Since $\mathcal{M}$ has eigenvalues $|p|^2, 0,\ldots$,
we see that
 \begin{equation*}
    d_0u\wedge d_1u= \sum_{A,B=0 }^{2n-1}Z_{A0'}u Z_{B1'}u\,\omega^A \wedge  \omega^B = \sum_{A,B=0 }^{2n-1}\widetilde{\triangle}_{ AB  } \,\omega^A \wedge  \omega^B.
 \end{equation*} is elementary strongly positive by Proposition \ref{prop:normal-form}.
\end{proof}
See \cite[Proposition 3.3]{WZ} for this proposition for $\mathbb{H}^n$ with a different proof.

\subsection{The closed   strongly positive
$2$-form given by a smooth PSH}

\begin{prop}\label{prop:hyperhermitian-nonnegative} For $u\in   C^2(\Omega)$, $u$ is PSH if and only if  the  hyperhermitian matrix
$(\overline{Q_l} {Q_m} u - 8\delta_{lm}\mathbf{i}\partial_tu )$ is  nonnegative.
\end{prop}
The tangential mapping
$\iota_{\eta,q *} $ maps horizontal  left invariant vector fields     on $\widetilde{\mathscr{H}}_q  $ to that on the
quaternionic  Heisenberg line $\mathscr{H}_{\eta,q }  $. In particular, we have
\begin{prop}\label{prop:iotaY} For $q\in \mathbb{H}^n\setminus \mathfrak{D}$,
\begin{equation}\label{eq:iotaY}
\iota_{\eta,q*} \widetilde{X_j }= \sum_{l=0 }^{n-1}\sum_{k=1}^4\left (\overline{q_l}^{\mathbb{R}}\right)_{j k} X_{4l+k}
\end{equation}
\end{prop}
\begin{proof} Since $\iota_{\eta,q}=\tau_\eta\circ\iota_{q}$ and $X_j$'s are invariant under $\tau_\eta$, it sufficient to prove (\ref{eq:iotaY}) for $\eta=0$.
For fixed $j=1,2,3,4$ and $l=1,\ldots, n$, note that
\begin{equation*}
   \widehat{q_l\mathbf{i}_j}=  q_l^{\mathbb{R}}\left(\begin{array}{c}  \vdots \\1\\ \vdots\end{array}\right)= \left(\begin{array}{c} \left(q_l^{\mathbb{R}}\right)_{1j}\\ \vdots\\\left(q_l^{\mathbb{R}}\right)_{4j}\end{array}\right)  .
\end{equation*}by (\ref{eq:widetildeX-q}).
  Thus for $q=(q_1,\ldots,q_n)\in \mathbb{H}^n$ and $\varsigma\in \mathbb{R}$, if we write $\iota_{q } (\lambda,t)=(q\lambda,t)=(x,t)$,  we get
  \begin{equation*}
   \begin{split}
 \iota_{q } \{(\lambda,t)(  \varsigma\mathbf{i}_j  ,0)\}&=  \left(q(\lambda+ \varsigma\mathbf{i}_j ) ,t+2\varsigma \sum_{ k=1}^{4 }B^q_{kj} \lambda_k \right)\\&
    =  \left( \ldots,x_{4l+i}+\varsigma \left(q_l^{\mathbb{R}}\right)_{ij},  \ldots,
    t +2\varsigma \sum_{l=0}^{n-1}\sum_{ k,i=1}^{4 }J_{ki} x_{4l+k}(q_l^{\mathbb{R}})_{ij} \right) ,
    \end{split}
\end{equation*}
by the multiplication (\ref{eq:Hq0}) of the group $\widetilde{\mathscr{H}}_q$ and $B^q$ in (\ref{eq:Bq}). So
\begin{equation*}
\begin{split}
  \left (\iota_{q*} \widetilde{X_j}f\right) ( x,t )&=\left.\frac d{d\varsigma}\right |_{\varsigma=0} f\left( \iota_{q } \{(\lambda,t)(  \varsigma\mathbf{i}_j  ,0)\}\right)
   \\& = \sum_{l=0 }^{n-1}\sum_{i=1}^4 \left(q_l^{\mathbb{R}}\right)_{ij}  \frac {\partial f}{\partial
   x_{4l+i}}+2 \sum_{l=0 }^{n-1}\sum_{k,i=1}^4J_{ki} x_{4l+k}\left(q_l^{\mathbb{R}}\right)_{ i j} \frac {\partial
   f}{\partial t }
  \\& =\sum_{l=0 }^{n-1}\sum_{i=1}^4 \left(\overline{q_l}^{\mathbb{R}}\right)_{j i}  X_{4l+i} f(x, t)
\end{split}
\end{equation*}by  (\ref{eq:iso}).
\end{proof}

{\it Proof of Proposition \ref{prop:hyperhermitian-nonnegative}}.
Denote
$
   \overline{\widetilde{Q}}:= \widetilde{X_1} +\mathbf{i}\widetilde{X_{ 2}}+\mathbf{j}\widetilde{X_{ 3 }}+\mathbf{k}\widetilde{X_{ 4}} .
$
Then we have
\begin{equation}\label{eq:qQ}
  \iota_{\eta,q*} \overline{\widetilde{Q}}= \sum_{j =1}^4 \iota_{\eta,q*}  \widetilde{X_j}\mathbf{i}_j= \sum_{l=0 }^{n-1}\sum_{j,k=1}^4\left(\overline{q_l}^{\mathbb{R}}\right)_{j k} X_{4l+k}\mathbf{i}_j =
   \sum_{l=0 }^{n-1} \overline{q}_l\overline{Q_l}
\end{equation}
by Proposition \ref{prop:iotaY} and definition of ${q }^{\mathbb{R}}$ in (\ref{eq:X-R}), and
$
  \iota_{\eta,q*}  {\widetilde{Q}}=
   \sum_{l=0 }^{n-1} {Q_l} {q}_l
$ by taking conjugate.
Therefore for real $u$, we have
\begin{equation}\label{eq:QQ}
\iota_{\eta,q*}\left( \widetilde{X}_{ 1}^2+ \widetilde{X}_{ 2}^2+\widetilde{X}_{ 3}^2+ \widetilde{X}_{ 4}^2\right)u= {\rm Re}\left(
\iota_{\eta,q*}\overline{\widetilde{Q}}\cdot\iota_{\eta,q*} {\widetilde{Q}} u  \right)=    {\rm Re} \left(  \sum_{l,m=0 }^{n-1}  \overline{q}_l\cdot\overline{Q_l} {Q_m} u
\cdot  {q}_m\right).
\end{equation}

On the other hand, we have
\begin{equation*}
\sum_{l,m=0 }^{n-1}  \overline{q_l}\left (\overline{Q_l} {Q_m} u + 8\delta_{lm}\mathbf{i}\partial_tu\right ) {q}_m=\left(\sum_{l=0 }^{n-1}
\overline{q_l } \overline{Q_l}\right)\left(\sum_{m=0 }^{n-1}  {Q_m}{q}_m\right) u+ 8 \sum_{l=0 }^{n-1}  \overline{q_l} \mathbf{i}{q}_l \partial_tu.
\end{equation*}
Since the horizontal  quaternionic Hessian  $(\overline{Q_l} {Q_m} u +8
\delta_{lm}\mathbf{i}\partial_t u )$ is hyperhermitian  by Proposition \ref{prop:hyperhermitian}, we see that the above quadratic form is real for any $q$. Note that $ \overline{p} \mathbf{i}p\in {\rm Im}\,\mathbb{ H}$ for any $0\neq p\in \mathbb{ H}$.
Therefore, we get
\begin{equation}\label{eq:quadratic}\begin{split}\sum_{l,m=0 }^{n-1}   \overline{q_l} \left(\overline{Q_l} {Q_m} u+ 8\delta_{lm}\mathbf{i}\partial_tu\right )
{q}_m&=
 {\rm Re} \left(  \sum_{l,m=0 }^{n-1}   \overline{q_l}\cdot\overline{Q_l} {Q_m} u \cdot  {q}_m\right)
 =(\iota_{\eta,q*}\widetilde{\triangle_q} )u
\end{split}\end{equation} for $q\in \mathbb{H}^n\setminus \mathfrak{D}$ by (\ref{eq:QQ}).

Now if  $u$ is PSH, then $ \widetilde{\triangle_q} (\iota^*_{\eta,q }u)$ is
nonnegative by applying Theorem  \ref{thm:meanvalue} to the group $\widetilde{\mathscr{H}}_q$ for $q\in \mathbb{H}^n\setminus \mathfrak{D}$.  Consequently,  (\ref{eq:quadratic}) holds for any
$q\in \mathbb{H}^n$ by continuity, i.e.  the  hyperhermitian matrix  $(\overline{Q_l} {Q_m} u + 8\delta_{lm}\mathbf{i}\partial_tu )$ is  nonnegative. Conversely, if   the  hyperhermitian matrix    is  nonnegative, we
get $u$ is is  subharmonic on each   quaternionic Heisenberg line $\mathscr{H}_{\eta,q}$ for any $q\in \mathbb{H}^n\setminus \mathfrak{D}$ and
$\eta\in \mathscr{H}^n$ by applying Theorem  \ref{thm:meanvalue} again. \hskip 80mm$\Box$

\begin{cor}\label{cor:hyperhermitian-positive} For $u\in PSH\cap C^2(\Omega)$,  $\triangle u$ is a closed   strongly positive
$2$-form.
\end{cor}
\begin{proof}  It follows from applying Proposition \ref{prop:normal-form} to nonnegative $\mathcal{M}=(\overline{Q_l} {Q_m} u - 8\delta_{lm}\mathbf{i}\partial_tu )$ and using (\ref{eq:Q-delta'}).
 \end{proof}

\begin{cor} A $C^2$ function $u$ is
pluriharmonic if and only if $\triangle u=0$.
\end{cor}
\begin{proof} $u$ is
pluriharmonic means that $\widetilde{\triangle_q}\iota^*_{\eta,q}u=0$ on
each   quaternionic Heisenberg line $\widetilde{\mathscr{H}}_{q}$ for any $\eta\in\mathscr{H}  $ and $ q\in\mathbb{H}^n\setminus  \mathfrak{D}  $. It holds if and only if
\begin{equation*}\sum_{l,m}  \overline{q}_l (\overline{Q_l} {Q_m} u+ 8\delta_{lm}\mathbf{i}\partial_tu ) {q}_m=0
\end{equation*} for any $q\in \mathbb{H}^n$ by
(\ref{eq:quadratic}), i.e. $(\overline{Q_l} {Q_m} u + 8\delta_{lm}\mathbf{i}\partial_tu )  =0 $, which equivalent to $\triangle u=0$ by
(\ref{eq:Q-delta'}).
\end{proof}

Recall that the  {\it tangential $1$-Cauchy-Fueter operator} on a domain $\Omega$ in the   Heisenberg group  $\mathscr{H} $
is
 $\mathscr D: C^1(\Omega,\mathbb{C}^2)\rightarrow C^0(\Omega,\mathbb{C}^{2n})$ given by
 \begin{equation*}
   ( \mathscr D f)_A=\sum_{A'=0',1'} Z_A^{A'}f_{A'},\qquad, A=0,\ldots, 2n-1,
 \end{equation*}
 where $Z_A^{0'}=Z_{A1'}$ and $Z_A^{1'}=-Z_{A0'}$. A $\mathbb{C}^2$-valued function $f=(f_{0'},f_{1'})=(f_1 +\mathbf{i}f_2,f_3
 +\mathbf{i}f_4)$ is called {\it $1$-CF} if
 $\mathscr D f=0$.

 \begin{prop} \label{cor:pluriharmonic}       Each real component    of a   $1$-CF function $f:  \mathscr
 H^n\rightarrow \mathbb C^2$ is pluriharmonic.
\end{prop}
\begin{proof} Note that $\sum_{A'=0',1'}Z_A^{A'}f_{A'}=0$ is equivalent to $\sum_A\sum_{A'=0',1'}Z_A^{A'}f_{A'}\omega^A=0$, which can be
written as
 \begin{equation*}
    d_1 f_{0'}- d_{0}f_{1'}=0.
 \end{equation*}
 Apply $ d_0$ on both sides to get $d_0  d_{1}f_{0'}=0$ since $d_0^2=0$. Similarly, we get $d_0  d_{1}f_{1'}=0$. Writing
 $f_{0'}=f_1+\mathbf{i}f_2$ for some real functions $f_1$ and $f_2$, we have
 \begin{equation*}
   \triangle f_1+\mathbf{i}\triangle f_2=0.
 \end{equation*}
 Note that for a real valued function $u$,  $\triangle u$ is a real $2$-form $ \mathbf{(proof)} $, i.e. $\rho(\textbf{j})\triangle u=\triangle u$. We get
 \begin{equation*}
   \triangle f_1-\mathbf{i}\triangle f_2=0.
 \end{equation*} Thus $ \triangle f_1=0=\triangle f_2$.  Similarly, we have $ \triangle f_3=0=\triangle f_4$.
 \end{proof}
See Corollary 2.1 in \cite{wang1} for this Proposition on the quaternionic space $\mathbb{H}^n$. Since $1$-regular functions are
abundant,
so are pluriharmonic functions on the   Heisenberg group.

\subsection{Closed positive currents}

An element of the dual space
$(\mathcal {D}^{2n-p}(\Omega))' $ is called a~\emph{$p$-current}. A $2k$-current $T$ is said to be \emph{positive} if we have
$T(\eta)\geq0$ for any strongly positive form $\eta\in\mathcal {D}^{2n-2k}(\Omega)$.
Although a $2n$-form is not an authentic differential form and we cannot integrate it, we can define
\begin{equation}\label{2.24}\int_\Omega F:=\int_\Omega f dV,
\end{equation}if we write $F=f~\Omega_{2n}\in L^1(\Omega,\wedge^{2n}\mathbb{C}^{2n})$,
where $dV$ is the Lebesgue measure.
In general, for a $2n$-current $F=\mu~\Omega_{2n}$ with the coefficient to be a measure $\mu$, define
\begin{equation}\label{2.273}\int_\Omega F:=\int_\Omega \mu.\end{equation}

Now for the $p$-current $F$, we define a $(p+1)$-current $d_\alpha F$ as
\begin{equation}\label{eq:generalized-sense}(d_\alpha F)(\eta):=-F(d_\alpha\eta),\qquad \alpha=0,1,
\end{equation}for any test $(2n-p-1)$-form $\eta$. We say a current $F$ is \emph{closed} if
$d_0F=d_1F=0 .$

An element  of the dual space
$(\mathcal {D}_0^{2n-p}(\Omega))' $ are called a \emph{$p$-current  of
order zero}. Obviously, the $2n$-currents are just the
distributions on $\Omega$, whereas the $2n$-currents of order
zero are Radon measures on $\Omega$.
 Let $\psi$ be a $p$-form whose coefficients are
locally integrable in $\Omega$. One can associate with $\psi$ the $p$-current $T_{\psi}$ defined by
\begin{equation*}T_{\psi}(\varphi)=\int_\Omega\psi\wedge\varphi,\quad{\rm
for\hskip 2mm any } \quad\varphi\in \mathcal {D}^{2n-p}(\Omega).
\end{equation*}

If $T$
is a $2k$-current on $\Omega$, $\psi$ is a $2l$-form on
$\Omega$ with coefficients in $C^{\infty}(\Omega)$, and $k+l\leq n$,
then the
formula\begin{equation}\label{2.27}(T\wedge\psi)(\varphi)=T(\psi\wedge\varphi)
 \qquad\text{for}~~~\varphi\in \mathcal {D}^{2n-2k-2l}(\Omega)
\end{equation}defines a $(2k+2l)$-current. In particular, if $\psi$ is a smooth function, $\psi T(\varphi)=T(\psi\varphi)$.\par

A $2k$-current $T$ is said to be \emph{positive} if we have $T(\eta)\geq0$ for any $\eta\in
C_0^\infty(\Omega,SP^{2n-2k}\mathbb{C}^{2n})$. In other words, $T$ is positive if for any $\eta\in
C_0^\infty(\Omega,SP^{2n-2k}\mathbb{C}^{2n})$,   $T\wedge\eta=\mu~\Omega_{2n}$ for some    positive
distribution $\mu$ (and hence a measure).

Let $I=(i_1,\ldots,i_{2k})$ be a multi-index such that $1\leq
i_1<\ldots< i_{2k}\leq n$. Denote by $\widehat{I}=(l_1,\ldots,l_{2n-2k})$ the
\emph{increasing complements} to $I$ in the set $\{0,1,\ldots,2n-1\}$, i.e.,
$\{i_1,\ldots,i_{2k}\}\cup\{l_1,\ldots,l_{2n-2k}\}=\{0,1,\ldots,2n-1\}$. For a $2k$-current $T$ in $\Omega$ and multi-index $I$, define
distributions $T_I$ by
$T_I(f):=\varepsilon_IT(f\omega^{\widehat{I}})$
for $f\in
C_0^\infty(\Omega)$, where $\varepsilon_I=\pm1$ is so chosen that
\begin{equation}\label{2.26}\varepsilon_I\omega^I\wedge \omega^{\widehat{I}}=\Omega_{2n}.
\end{equation} If $T$ is a current of order $0$, the
distributions $T_I$ are Radon measures and
\begin{equation}\label{3.91}T(\varphi)=\sum_I\varepsilon_IT_I(\varphi_{\widehat{I}}),
\end{equation}for $\varphi=\sum_{\widehat{I}}\varphi_{\widehat{I}}\omega^{\widehat{I}}\in \mathcal {D}^{2n-2k}(\Omega)$,
where $I$ and $\widehat{I}$ are increasing. Namely, \begin{equation}\label{eq:coef-T}
     T=\sum_IT_I\omega^I,
\end{equation}
  where the
summation is taken over increasing multi-indices of length
$2k$, holds in the sense that if we
 write $T\wedge\varphi=\mu~\Omega_{2n}$ for some Radon measure $\mu$, then  we have
\begin{equation}\label{3.92}T(\varphi)=\int_\Omega \mu=\int_\Omega T\wedge\varphi.
\end{equation}

\begin{prop}\label{p3.13}Any positive $2k$-current $T$ on $\Omega$ has measure
coefficients (i.e. is of order zero), and we can write $T=\sum_IT_I\omega^I$ for some complex Radon measures $T_I$, where the
summation is taken over all increasing multi-indices $I $.\end{prop}
\begin{proof}By Proposition  5.4 in \cite{alesker2}, we can find $\{\varphi_L\}\subseteq SP^{2n-2k}\mathbb{C}^{2n}$ such that any
$\eta\in\wedge^{2n-2k}\mathbb{C}^{2n}$ is a $\mathbb{C}$-linear combination of $\varphi_L$, i.e., $\eta=\sum\lambda_L\varphi_L$ for some
$\lambda_L\in\mathbb{C}$. Let $\{\widetilde{\varphi_L}\}$ be a basis of $\wedge^{2k}\mathbb{C}^{2n}$ which is dual to $\{\varphi_L\}$.
  Then $T=\sum T_L\widetilde{\varphi_L}$ with  distributional coefficients $T_L$  as
(\ref{eq:coef-T}). If $\psi$ is a nonnegative test function, $\psi\varphi_L\in C_0^\infty(\Omega,SP^{2n-2k}\mathbb{C}^{2n})$. Then
$T_L(\psi)=T(\psi\varphi_L)\geq0$ by definition. It follows that $T_L$ is a positive distribution, and so is a nonnegative measure.
\end{proof}

The following Proposition is obvious and will be used frequently.
\begin{prop}(1) (linearity) For $2n$-currents $T_1$ and $T_2$ with (Radon) measure coefficients, we have
$$\int_\Omega\alpha T_1+ \beta T_2=\alpha\int_\Omega T_1+ \beta \int_\Omega T_2.$$
(2) If $T_1\leq T_2$ as positive $2n$-currents (i.e. $\mu_1\leq\mu_2$ if we write $T_j=\mu_j\Omega_{2n}$, $j=1,2$), then $\int_\Omega
T_1\leq \int_\Omega T_2$.
\end{prop}

\begin{lem}\label{p3.9} {\rm (Stokes-type formula)}   Let $\Omega$ be a bounded   domain with smooth boundary and defining function $\rho$ (i.e. $\rho=0$ on $\partial \Omega$ and $\rho<0$ in $\Omega$) such that $|{\rm grad} \rho|=1$.
Assume that $T=\sum_AT_A\omega^{\widehat{A}}$ is a smooth $(2n-1)$-form in $\Omega$,
where $\omega^{\widehat{A}}=\omega^A  \rfloor \Omega_{2n}$. Then for   $h\in C^1(\overline{\Omega})$, we have
\begin{equation}\label{stokes}\int_\Omega hd_\alpha T=-\int_\Omega d_\alpha h\wedge T+\sum_{A=0}^{2n-1} \int_{\partial\Omega}hT_A
 Z_{A\alpha'}\rho\, dS,\end{equation}
where    $dS$ denotes the surface measure of $\partial\Omega$. In particular, if $h=0$ on
$\partial\Omega$, we have \begin{equation}
\label{eq:stokes0}\int_\Omega hd_\alpha T=-\int_\Omega d_\alpha h\wedge T,\qquad \alpha=0,1,\end{equation}
\end{lem}
\begin{proof}Note that \begin{equation*}d_\alpha(hT)=\sum_{B,A} Z_{B\alpha'}(hT_A)\omega^B\wedge\omega^{\widehat{A}}=\sum_{A} Z_{A\alpha'
}(hT_A) \Omega_{2n}.\end{equation*} Then
 \begin{equation*}\int_\Omega d_\alpha(hT)=\int_\Omega\sum_{A} Z_{A\alpha' }(hT_A) dV=\int_{\partial\Omega}\sum_{A} hT_A
 Z_{A\alpha' }\rho~dS,\end{equation*} by definition (\ref{2.24}) and integration by part,
 \begin{equation}\label{eq:Stokes}
    \int_\Omega  X_j f \, dV=\int_{\partial\Omega} f X_j \rho\,dS,
 \end{equation}
 for $j=1,\ldots 4n$. (\ref{eq:Stokes}) holds because the coefficient  of $\partial_t$ is independent of $t$.
  (\ref{stokes}) follows from the above formula and
 $d_\alpha(hT)=d_\alpha h\wedge T+hd_\alpha T$.
\end{proof}

Now let us show that $d_\alpha F$  in the generalized sense (\ref{eq:generalized-sense}), coincides with the original definition  when $F$ is a smooth $2k$-form. Let $\eta$ be arbitrary $(2n-2k-1)$-test form   compactly supported  in $\Omega$.
It follows from Lemma \ref{p3.9} that $\int_\Omega d_\alpha(F\wedge\eta)=0.$ By Proposition \ref{prop:d2} (3),
$d_\alpha(F\wedge\eta)=d_\alpha F\wedge\eta+F\wedge d_\alpha\eta$. We have
\begin{equation}\label{3.121}-\int_{\Omega}F\wedge
d_\alpha\eta=\int_{\Omega}d_\alpha F\wedge\eta,\qquad\qquad~i.e.,\qquad(d_\alpha F)(\eta)=-F(d_\alpha\eta).\end{equation}
 We also define $\triangle F$ in the
generalized sense, i.e., for each test $(2n-2k-2)$-form $\eta$,
\begin{equation}\label{3.8}
(\triangle
F)(\eta):=F(\triangle\eta).\end{equation}

As a corollary, $\triangle F$ in the generalized sense   coincides with the original definition   when $F$ is a smooth $2k$-form:
\begin{equation*}
   \int \triangle F\wedge
 \eta = \int F \wedge \triangle\eta.
\end{equation*}

\begin{cor}\label{p3.2} For $u\in PSH(\Omega)$, $\triangle u$ is
a closed positive $2$-current.
\end{cor}
\begin{proof} If $u$ is smooth, $\triangle u$ is a closed  strongly  positive
$2$-form by   Corollary \ref{cor:hyperhermitian-positive}. When $u$ is not smooth, consider
regularization $u_{\varepsilon}=\chi_{\varepsilon}* u$ as in  Proposition \ref{prop:PSH-property} (6).   It suffices to
show that the coefficients $\triangle_{AB}u_{\varepsilon}\rightarrow
\triangle_{AB}u$ in the sense of weak convergence of distributions.
For any $\varphi\in C_0^{\infty}(\Omega)$,
\begin{equation*}\begin{aligned}\int\triangle_{AB}u_{\varepsilon}\cdot\varphi
=\int u_{\varepsilon}\cdot\triangle_{AB}\varphi\rightarrow
\int u\cdot\triangle_{AB}\varphi=(\triangle_{AB}u)(\varphi)
\end{aligned}\end{equation*}as $\varepsilon\rightarrow0$, by using integration by part (\ref{eq:Stokes}) and the standard fact that $\chi_{\varepsilon}* u\rightarrow u$ in $L^1_{\rm loc}(\Omega)$ if $u \in L^1_{\rm loc}(\Omega)$ \cite{Ma}. It follows that the currents
$\triangle u_{\varepsilon}$ converge to $\triangle u$, and so the
current $\triangle u$ is positive. For any test form $\eta$,
\begin{equation*}(d_\alpha\triangle u)(\eta)=-\triangle
u(d_\alpha\eta)=-\lim_{\varepsilon\rightarrow0}\triangle
u_\varepsilon(d_\alpha\eta)=\lim_{\varepsilon\rightarrow0}(d_\alpha\triangle
u_\varepsilon)(\eta)=0,\end{equation*}$\alpha=0,1,$ where the last identity
follows from Corollary \ref{p2.33}. Here $u_\varepsilon$ is
smooth, and $d_\alpha\triangle u_\varepsilon$ coincides with its usual
definition.\end{proof}

\section{The quaternionic Monge-Amp\`{e}re measure   over the
Heisenberg group}

For      positive $(2n-2p)$-form  $T$ and
  an arbitrary compact subset $K$, define
$\|T\|_K:=\int_K T\wedge\beta_n^p,$ where
$\beta_n$ is given by (\ref{eq:bn}). In
particular, if $T$ is a positive $2n$-current,
$\|T\|_K$ coincides with $\int_KT$ defined by
(\ref{2.273}). Let $\|\cdot\|$ be a norm on $\wedge^{2k}\mathbb{C}^{2n}$.

\begin{lem}\label{l4.1}{\rm (Lemma 3.3 in \cite{wan-wang})} For  $\eta\in\wedge^{2k}_{\mathbb{R}}\mathbb{C}^{2n}$ with $\|\eta\|\leq 1$,
$\beta_n^k\pm\varepsilon\eta$ is a positive $2k$-form for some sufficiently small
$\varepsilon>0$.
\end{lem}

\begin{prop}\label{prop:estimate-C0} (Chern-Levine-Nirenberg type estimate) Let
$\Omega$ be a domain in $\mathscr{H}^n$. Let $K$ and $L$ be compact subsets
of $\Omega$ such that $L$ is contained in the interior of $K$. Then
there exists a constant $C$ depending only on $K,L$ such that for
any $u_1,\ldots u_k\in PSH(\Omega)\cap C^2(\Omega)$, we have
\begin{equation}\label{3.14}\|\triangle
u_1\wedge\ldots\wedge\triangle u_k \|_{L}\leq
C\prod_{i=1}^k\|u_i\|_{C^0(K)}.
\end{equation}  \end{prop}
\begin{proof} By Corollary \ref{cor:hyperhermitian-positive}, $\triangle
u_1\wedge\ldots\wedge\triangle u_k $ is already   closed and strongly positive.   Since $L$ is compact, there is a covering of $L$ by a
family of balls $D_j'\Subset D_j\subseteq K$. Let $\chi\geq0$ be a smooth function equals to 1 on $\overline{D_j'}$ with support in
$D_j$. For a closed smooth  $(2n-2p)$-form $T$, we have
\begin{equation}\label{eq:int-part}\begin{aligned} \int_\Omega\chi \triangle
u_1\wedge\ldots\wedge\triangle u_p \wedge T&=-\int_\Omega d_0\chi\wedge d_1 u_1\wedge\triangle
u_2\wedge\ldots\wedge\triangle u_p \wedge T\\&
=  -\int_\Omega u_1 d_1 d_0\chi\wedge\triangle
u_2\wedge\ldots\wedge\triangle u_p \wedge T\\&=   \int_\Omega u_1 \triangle\chi\wedge\triangle
u_2\wedge\ldots\wedge\triangle u_p \wedge T\end{aligned}\end{equation}by using
Stokes-type formula  (\ref{eq:stokes0}) and Proposition \ref{prop:d-delta}.
Then
 \begin{equation*}\begin{aligned}\|\triangle
u_1 \wedge\ldots\wedge\triangle u_k\|_{L\cap\overline{D_j'}}=&
\int_{L\cap\overline{D_j'}}\triangle
u_1 \wedge\ldots\wedge\triangle u_k\wedge\beta_n^{n-k}\leq \int_{D_j}\chi\triangle
u_1\wedge  \ldots\wedge\triangle u_k\wedge\beta_n^{n-k}\\=&\int_{D_j}u_1\triangle\chi
\wedge\triangle
u_2\wedge\ldots\wedge\triangle u_k\wedge\beta_n^{n-k}\\\leq& \frac 1\varepsilon
 \|u_1\|_{L^{\infty}(K)}\|\triangle\chi \| \int_{D_j}
 \triangle
u_2\wedge\ldots\wedge\triangle u_k\wedge\beta_n^{n-k+1},
\end{aligned}\end{equation*}by using (\ref{eq:int-part}) and   Lemma \ref{l4.1}.
The result follows by repeating this procedure.
\end{proof}

{\it Proof of Theorem \ref{thm:MA-measure}}. It is sufficient to prove for any compactly supported continuous function $\chi$, the sequence
$
   \int_\Omega\chi(\triangle u_j)^n
$
is a Cauchy sequence. We can assume $\chi\in C_0^\infty (\Omega)$.
Note the following identity
\begin{equation}\label{eq:difference}\begin{split}
   (\triangle v)^n-  (\triangle u )^n   &=   \sum_{p=1}^n  \left\{   ( \triangle
v)^{p } \wedge (\triangle u )^{n-p }- ( \triangle
v)^{p-1} \wedge (\triangle u )^{n-p +1}\right\}\\&=\sum_{p=1}^n  ( \triangle
v)^{p-1} \wedge \triangle \left(v -u  \right)\wedge  (\triangle u )^{n-p }.
\end{split}\end{equation} Then we have
\begin{equation*}\begin{split}
 \left  | \int_\Omega\chi(\triangle u_j)^n\right.-&  \left.\int_\Omega\chi(\triangle u_k)^n \right |  \leq \sum_{p=1}^n\left |\int_K \chi \triangle
u_j  \wedge\ldots\wedge\triangle \left(u_j -u_k \right)\wedge\triangle
u_{k }  \wedge\ldots\wedge\triangle u_k\right| \\
=&\sum_{p=1}^n \left|\int_K \left(u_j -u_k \right)\triangle
u_j  \wedge\ldots\wedge\triangle\chi\wedge\triangle
u_{k }  \wedge\ldots\wedge\triangle u_k\right |\\
\leq &\frac {\|\triangle\chi \|}\varepsilon \left\|u_j -u_k \right\|_\infty\sum_{p=1}^n \int_K\triangle
u_j  \wedge\ldots\wedge  \beta_n \wedge\triangle
u_{k }  \wedge\ldots\wedge\triangle u_k
 \leq   C  \left\|u_j -u_k \right\|_\infty .
\end{split}\end{equation*}
as in the proof of Proposition \ref{prop:estimate-C0}, where $C$ depends on the uniform upper bound of $\left\|u_j   \right\|_\infty$.

\begin{prop}\label{prop:ineq}  Let  $u ,v\in C (\overline{\Omega})$ be    plurisubharmonic functions. Then
$
    (\triangle(u+v))^n\geq (\triangle u )^n+(\triangle v )^n.
$
 \end{prop}
\begin{proof} For smooth PSH
$u_{\varepsilon}=\chi_{\varepsilon}* u$, we have
\begin{equation*}
   (\triangle(u_{\varepsilon}+v_{\varepsilon}))^n= (\triangle u_{\varepsilon} )^n+(\triangle v_{\varepsilon} )^n+\sum_{j=1}^{n-1} C_n^j(\triangle u_{\varepsilon} )^j\wedge(\triangle v_{\varepsilon} )^{n-j}\geq (\triangle u_{\varepsilon} )^n+(\triangle v_{\varepsilon} )^n.
\end{equation*}
The result follows by taking limit $\varepsilon\rightarrow 0$ and using the convergence of the quaternionic Monge-Amp\`{e}re measure
   in Theorem \ref{thm:MA-measure}.
 \end{proof}

We need the following proposition to prove the minimum principle.
\begin{prop}\label{prop:compare}
 Let $\Omega$ be a bounded   domain with smooth boundary in $\mathscr{H} $, and let $u ,v\in C^2(\overline{\Omega})$ be    plurisubharmonic functions on $\Omega$. If $u=v$ on $\partial\Omega$ and $u\geq v$ in $\Omega$, then
 \begin{equation}\label{eq:compare}
    \int_\Omega (\triangle u)^n\leq  \int_\Omega (\triangle v)^n.
 \end{equation}
\end{prop}
\begin{proof} We have
\begin{equation}\label{eq:boundary-term}\begin{split}
  \int_\Omega (\triangle v)^n-\int_\Omega (\triangle u )^n   = &  \sum_{p=1}^n \int_\Omega d_0\left\{d_1 \left(v -u  \right)\wedge  ( \triangle
v)^{p-1} \wedge (\triangle u )^{n-p }\right \}   \\=&\sum_{p=1}^n\sum_{A=0}^{2n-1} \int_{\partial\Omega} T_A^{p}
\cdot Z_{A0'}\rho \cdot dS
\end{split}\end{equation}by using (\ref{eq:difference}) and
Stokes-type formula  (\ref{stokes}), if we  write
\begin{equation*}
   d_1 \left(v -u  \right)\wedge  ( \triangle
v )^{p-1} \wedge (\triangle u )^{n-p }=:\sum_AT_A^{p}\,\omega^{\widehat{A}},
\end{equation*}where $\rho$ is a defining function of $\Omega$ with  $|{\rm grad} \rho|=1$, and  $\omega^{\widehat{A}}=\omega^A  \rfloor \Omega_{2n}$. Note that we have
\begin{equation}\label{eq:boundary-term'}
  \sum_{A=0}^{2n-1}  T_A^{p}
\cdot Z_{A0'}\rho (\xi)\cdot\Omega_{2n} =d_0\rho(\xi)\wedge d_1 \left(v -u  \right)\wedge  ( \triangle
v )^{p-1} \wedge (\triangle u )^{n-p }.
\end{equation} Since $u=v$ on $\partial\Omega$ and $u\geq v$ in $\Omega$,  for a point $\xi\in \partial\Omega$ with ${\rm grad} (v-u)(\xi)\neq 0$ ,
  we can write  $ v-u = h \rho$ in a neighborhood of $\xi$ for some positive smooth function $h$. Consequently, we have ${\rm grad} (v-u)(\xi)= h(\xi) {\rm grad} \rho$, and so  $Z_{A1'}(v-u)(\xi)= h(\xi)  Z_{A1'}\rho (\xi)$ on $\partial\Omega$. Thus,
\begin{equation*}
   d_0\rho(\xi)\wedge d_1 \left(v -u  \right)(\xi)=h(\xi)  d_0\rho(\xi)\wedge d_1 \rho(\xi),
\end{equation*}
which is   strongly positive  by Proposition \ref{prop:wedge-positive}. Moreover, both $\triangle
v $ and $\triangle u$ are     strongly positive for   $C^2$ plurisubharmonic functions $u$ and $v$ on $\Omega$ by Proposition \ref{cor:hyperhermitian-positive}. We find that the
 the right hand of  (\ref{eq:boundary-term'}) is a positive $2n$-form, and so the integrant in the right hand of   (\ref{eq:boundary-term}) on $\partial\Omega$  is nonnegative if ${\rm grad} (v-u)(\xi)\neq 0$, while if ${\rm grad} (v-u)(\xi)= 0$,   the integrant at $\xi $  in (\ref{eq:boundary-term}) vanishes. Therefore the difference in (\ref{eq:boundary-term}) is  nonnegative.
\end{proof}

The proof of the   minimum principle is similar to the complex case \cite{BT} and the quaternionic case \cite{alesker1}, but we need some modifications because we do not know whether the regularization  $ \chi_\varepsilon*u  $ of a PSH function $u$ on the the
Heisenberg group is decreasing as $\varepsilon\rightarrow 0+$.
\vskip 3mm
{\it Proof of Theorem \ref{thm:minimun}}. Without loss of generality, we may assume $\min_{ \partial {\Omega}} \{u-v\}=0$. Suppose that
there exists a point $(x_0,t_0)\in \Omega$ such that $u(x_0,t_0)<v (x_0,t_0)$. Denote $\eta_0=\frac 12[v (x_0,t_0)-u(x_0,t_0)]$. Then for each $0<\eta<\eta_0$, the set
$
   G(\eta):=\{(x,t)\in \Omega; u(x ,t )+\eta<v( x ,t )\}
$
is a non-empt, open, relatively compact subset of $\Omega$.  Now consider
 \begin{equation*}
   G(\eta,\delta):=\{(x,t)\in \Omega; u(x ,t )+\eta<v( x ,t )+\delta|x-x_0|^2\}.
\end{equation*}There exists an increasing function $\delta(\eta)$ such that $G(\eta,\delta)$ for $0<\delta<\delta(\eta)$ is a non-empt, open, relatively compact subset of $\Omega$. On the other hand, there exists small $\alpha(\eta,\delta)$  such that for $0<\alpha<\alpha(\eta,\delta)$, we have
$
\{\xi\in \Omega; {\rm dist} (\xi,\partial\Omega)>\alpha\}=   :\Omega_\alpha\supset G(\eta,\delta)
$ for $0<\delta<\delta(\eta/2)$, where ${\rm dist} (\xi,\zeta)=\|\xi^{-1}\zeta\|$.

We hope to apply Proposition \ref{prop:compare} to $G(\eta,\delta)$ to get a contradict, but its boundary may not be smooth. We need to regularize them. Recall that  $u_\varepsilon\rightarrow u$ and $v_\varepsilon\rightarrow v$ uniformly  as $\varepsilon\rightarrow 0+$ on any  compact subset of $\Omega$.
Define
\begin{equation*}
   G(\eta,\delta,\varepsilon):=\{(x,t)\in \Omega; u(x ,t )+\eta<v_\varepsilon( x ,t )+\delta|x-x_0|^2\},
\end{equation*}
which satisfies $ G(\eta,\delta,\varepsilon)\subset  G(3\eta/4,\delta)\subset  G( \eta/2,\delta)$ if $0<\varepsilon<\alpha(\eta,\delta)$ is sufficiently small, since $|v(x,t)-v_\varepsilon(x,t)|\leq\eta/4$ for
$(x,t)\in G(\eta/2, \delta )$. Now choose $\tau $ so small that
\begin{equation*}
   G(\eta,\delta,\varepsilon,\tau):=\{(x,t)\in \Omega; u_\tau(x ,t )+\eta<v_\varepsilon( x ,t )+\delta|x-x_0|^2\}
\end{equation*}is a non-empt, open, relatively compact subset of $\Omega$. At last we can choose positive numbers $ \eta_1<\eta_2, \delta_0  ,  \varepsilon_0, \tau_0$ such that for any
$\eta\in [\eta_1,\eta_2]   $, $0<\varepsilon<\varepsilon_0$, $0<\tau<\tau_0$,
$ G(\eta,\delta_0 ,\varepsilon,\tau)$ is a non-empt, open, relatively compact subset of $\Omega$.

For fixed $\varepsilon,\tau  $, by Sard's theorem,  almost all values of the $C^\infty$ function $v_\varepsilon( x ,t )+\delta_0 |x-x_0|^2  -u_\tau(x ,t )  $ are regular, i.e. $ G(\eta,\delta_0 ,\varepsilon,\tau)$ has smooth boundary for almost all $\eta$. Consequently,  we can take sequence of numbers $\tau_k\rightarrow0$ and $\varepsilon_k\rightarrow0$ such that $ G(\eta,\delta_0 ,\varepsilon_k,\tau_k)$ has a smooth boundary for each $k$ and almost all $\eta\in [\eta_1,\eta_2]$. Now
 apply
Proposition \ref{prop:compare} to the domain $ G(\eta,\delta_0 ,\varepsilon_k,\tau_k)$ to   get
\begin{equation}\label{eq:control1}\begin{split}
   \int (\triangle u_{\tau_k})^n&\geq \int (\triangle (v_\varepsilon+\delta_0 |x-x_0|^2) )^n\geq \int (\triangle  v_{\varepsilon_k}   )^n+\delta_0 ^n \int (\triangle|x-x_0|^2  )^n\\&
   =\int (\triangle  v_{\varepsilon_k}   )^n+4^n n!\delta_0 ^n vol(G(\eta,\delta_0 ,\varepsilon_k,\tau_k))
\end{split}\end{equation}by using Proposition \ref{prop:ineq} (2), where the integral are taken over $G(\eta,\delta_0 ,\varepsilon_k,\tau_k)$, and
\begin{equation*}
   (\triangle|x-x_0|^2)^n=\left(\sum_{l=0}^{n-1}\triangle_{l(n+l)}|x-x_0|^2\omega^l\wedge \omega^{n+l}\right)^n=4^n n!\Omega_{2n},
\end{equation*}by the expression of $\triangle_{l(n+l)}$ in (\ref{eq:brackets-Z'}).
Since $(\triangle u)^n\leq    (\triangle v)^n$ and  $\eta\rightarrow   (\triangle v)^n ( G(\eta ,\delta_0  ))$ is decreasing in $\eta$, we can choose  a continuous point  $\eta$ such that $ G(\eta,\delta_0 ,\varepsilon_k,\tau_k)$ has a smooth boundary.
 For any $\eta_1<\eta'<\eta<\eta''<\eta_2$,
 $G(\eta',\delta_0 )\supset G(\eta,\delta_0 ,\varepsilon_k,\tau_k)\supset G(\eta'',\delta_0 )$ for large $k$. So we have
\begin{equation}\label{eq:control2}\begin{split}
  \int_{G(\eta',\delta_0  )} (  \triangle u_{\tau_k}  )^n \geq  \int_{G(\eta'',\delta_0  )} (\triangle  v_{\varepsilon_k}   )^n+(4\delta_0 )^n n!  vol(G(\eta'',\delta_0  ))
\end{split}\end{equation} by (\ref{eq:control1}). Thus,
\begin{equation*}\begin{split}
   (\triangle  u )^n ( G(\eta',\delta_0  ))&\geq (\triangle  v  )^n ( G(\eta'',\delta_0  ))  +(4\delta_0 )^n n!  vol(G(\eta'',\delta_0  )),
\end{split}\end{equation*} by   convergence of quaternionic Monge-Amp\`{e}re
measures  by Theorem \ref{thm:MA-measure}. At the continuous point $\eta$,
 we have
\begin{equation*}\begin{split}
    (\triangle v)^n ( G(\eta ,\delta_0  ))&\geq  (\triangle v)^n( G(\eta ,\delta_0  ))  +(4\delta_0 )^n n! vol(G(\eta'',\delta_0  )).
\end{split}\end{equation*}This is a contradict since $G(\eta'',\delta_0  )$ is a nonempt  open  subset of $\Omega$ for $\eta''$ close to $\eta$.

 \end{document}